\documentclass[11pt,a4paper]{article}
\usepackage{amsmath,amsfonts,amssymb,amsthm,psfrag,epsf,epsfig}

\newtheorem{prop}{Proposition}[section]

\newtheorem{lm}[prop]{Lemma}

\newtheorem{thm}[prop]{Theorem}
\newtheorem{cor}[prop]{Corollary}

\newtheorem{rem}[prop]{Remark}

\newcommand{\convd}{\overset{d}{\underset{{n\rightarrow \infty}}{\longrightarrow}}}
\newcommand{\eq}{\begin{equation}}
\newcommand{\en}{\end{equation}}

\newcommand{\beq}{\begin{eqnarray*}}
\newcommand{\eeq}{\end{eqnarray*}}

\def\build#1_#2^#3{\mathrel{\mathop{\kern 0pt#1}\limits_{#2}^{#3}}}

\newcommand{\bP}{\mathbb{P}}

\newcommand{\bN}{\mathbb{N}}
\newcommand{\bR}{\mathbb{R}}

\newcommand{\bE}{\mathbb{E}}
\newcommand{\bK}{\mathbb{K}}
\newcommand{\bT}{\mathbb{T}}

\newcommand{\bX}{\mathbb{X}}

\newcommand{\cC}{\mathcal{C}}

\newcommand{\cG}{\mathcal{G}}
\newcommand{\cH}{\mathcal{H}}

\newcommand{\cR}{\mathcal{R}}
\newcommand{\cS}{\mathcal{S}}
\newcommand{\cT}{\mathcal{T}}

\newcommand{\cX}{\mathcal{X}}

\newcommand{\fs}{\mathbf{s}}

\newcommand{\bZ}{\mathbb{Z}}

\newcommand{\ft}{\mathbf{t}}

\newcommand{\modcut}{{\rm cut}_{\rm HW}}
\newcommand{\ourcut}{{\rm cut}_{\rm HW}^\circ}

\def\beqlb{\begin{eqnarray}}\def\eeqlb{\end{eqnarray}}
\def\beqnn{\begin{eqnarray*}}\def\eeqnn{\end{eqnarray*}}

\def\rar{\rightarrow}

\newcommand{\lambdaLaplace}{r}

\begin{document}

\title{\vspace{-0.7cm}
Gromov-Hausdorff-Prokhorov convergence of vertex cut-trees\\ of $n$-leaf Galton-Watson trees}
\author{
Hui He%
\thanks{%
Laboratory of Mathematics and Complex Systems, School of Mathematical Sciences, Beijing Normal University, Beijing 100875, P.R.China; email hehui@bnu.edu.cn}
\and
Matthias Winkel\thanks{%
Department of Statistics, University of Oxford, 1 South Parks Road, Oxford OX1 3TG, UK; email winkel@stats.ox.ac.uk}}

\maketitle

\vspace{-0.4cm}

\begin{abstract} In this paper we study the vertex cut-trees of Galton-Watson trees conditioned to have $n$ leaves. This notion is a slight variation of Dieuleveut's
  vertex cut-tree of Galton-Watson trees conditioned to have $n$ vertices. Our main result is a joint Gromov-Hausdorff-Prokhorov convergence in the finite variance case
  of the Galton-Watson tree and its vertex cut-tree to Bertoin and Miermont's joint distribution of the Brownian CRT and its cut-tree. The methods also apply to the
  infinite variance case, but the problem to strengthen Dieuleveut's and Bertoin and Miermont's Gromov-Prokhorov convergence to Gromov-Hausdorff-Prokhorov remains open
  for their models conditioned to have $n$ vertices.

\emph{AMS 2010 subject classifications: Primary 60J80; Secondary 60J25, 60F17.\newline
}
\end{abstract}

\section{Introduction}

Consider a \em rooted planar tree \em $(\ft,\rho)$. Specifically, $\ft$ consists of a finite vertex set $V(\ft)$ including the \em root \em $\rho\in V(\ft)$, a set
$E(\ft)$ of directed edges $u\rightarrow v$, one edge for each $u\in V(\ft)\setminus\{\rho\}$ without creating cycles, and a planar order, which we describe below. We
call $v$ the \em parent \em of $u$ and $k_v(\ft)=\#\{w\in V(\ft)\colon w\rightarrow v\}$ the \em number of children \em or \em degree \em of $v\in V(\ft)$. A vertex $v\in V(\ft)$ with
$k_v(\ft)=0$ is called a \em leaf\em. We denote by ${\rm Lf}(\ft)=\{v\in V(\ft)\colon k_v(\ft)=0\}$ the set of leaves of $\ft$, and by $\zeta(\ft)=\#V(\ft)$ and
$\lambda(\ft)=\#{\rm Lf}(\ft)$ the numbers of vertices and leaves, respectively. Non-leaf vertices, including the root, if $\zeta(\ft)\ge 2$, are called \em branch
points\em. The set of branch points is ${\rm Br}(\ft)=V(\ft)\setminus{\rm Lf}(\ft)$. The \em planar order \em specifies for each $v\in{\rm Br}(\ft)$ a total order on the
set of its $k_v(\ft)$ children. Unless otherwise stated, we will assume that $k_v(\ft)\neq 1$ for all $v\in V(\ft)$.
\begin{itemize}\item
Let $n=\lambda(\ft)$. We introduce \em our vertex splitting rule\em, as follows. Select a branch point at random, $v\in{\rm Br}(\ft)$ with probability
$(k_v(\ft)-1)/(n-1)$. Fragment the vertex set into $k_v(\ft)+1$ connected components by removing
the edges $w\rightarrow v$ from all the children $w$ of the selected branch point $v$. The component of $\rho$ now has $v$ as a leaf, while the $k_v(\ft)$ other components
are now rooted at the children of $v$.
We apply the splitting rule independently and repeatedly until all components are singleton leaves. We define \em our vertex cut-tree \em
$\ourcut(\ft)$ as the rooted planar tree taking as vertex set the set of components (subsets of $V(\ft)$) that ever exist, as edge relation the relation
between each component and its fragments, as root the initial single component ($V(\ft)$) that contains all vertices, and as planar order the order that has for the component split at $v$ the component of $v$ first and the other $k_v(\ft)$
components in the order their roots have in $\ft$ as children of $v$.
\end{itemize}
This is illustrated in Figure \ref{fig1}. Our notion of a cut-tree appears to be new, but is closely related to other cut-trees that have been studied and indeed have motivated us
for this work:
\begin{center}
\begin{figure}[t]\centering
\includegraphics[scale=0.57]{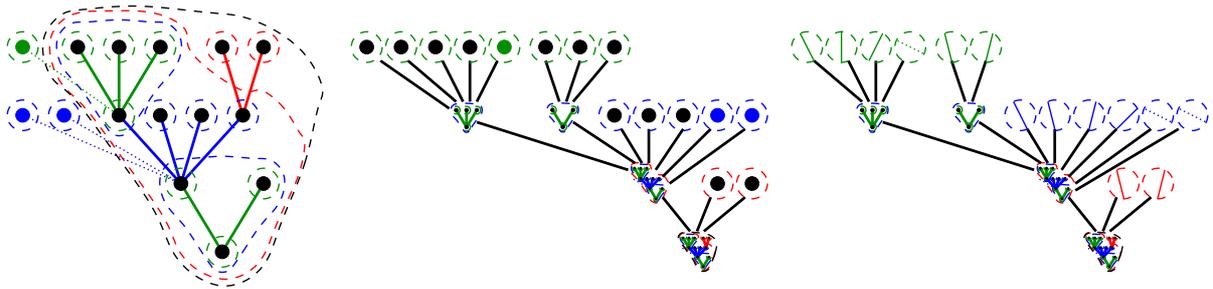}\vspace{-0.3cm}
  \caption{Illustration of $\mathbf{t}$, ${\rm cut}_{\rm HW}(\mathbf{t})$ and ${\rm cut}_{\rm D}(\widehat{\mathbf{t}})$ for $n=8$. Dotted lines capture components of ${\rm cut}_{\rm HW}(\ft)$. To see ${\rm cut}_{\rm HW}^\circ(\mathbf{t})$, omit the singleton components of ${\rm cut}_{\rm HW}(\mathbf{t})$ with coloured bullets.}
    \label{fig1}\vspace{-0.2cm}
\end{figure}
\end{center}
\begin{itemize}
  \item Let $n=\zeta(\ft)$. Meir and Moon \cite{MeirMoon} introduced an \em edge splitting rule\em, as follows. Select an edge uniformly at random. Remove the edge (as a singleton) and retain up to two further components (above/below). Pitman \cite{Pitman1999} and Bertoin \cite{Bertoin2012} studied the forest of components in connection to
    additive coalescents and forest fires. Bertoin and Miermont \cite{BM} introduced the associated \em edge cut-tree \em ${\rm cut}_{\rm BM}(\ft)$. In the case of
    finite-variance Galton-Watson trees conditioned to have $n$ vertices, they showed Gromov-Prokhorov (GP) convergence of tree and cut-tree to a pair
    $(\cT_{\rm Br},{\rm cut}(\cT_{\rm Br}))$ of Brownian Continuum Random Trees (CRTs).
  \item Let $n=\zeta(\ft)$. \em Dieuleveut's \cite{Die13} vertex splitting rule and vertex-cut tree \em ${\rm cut}_{\rm D}(\ft)$ are, as
    follows. Select $v\in{\rm Br}(\ft)$ with probability $k_v(\ft)/(n-1)$. Fragment the edge set into up to $2k_v(\ft)+1$
    components including all edges above the vertex as singletons. In the case of finite-variance Galton-Watson trees conditioned to have $n$ vertices,
    Dieuleveut showed GP convergence of the tree and her cut-tree to the same pair $(\cT_{\rm Br},{\rm cut}(\cT_{\rm Br}))$. She also obtained an
    infinite-variance result with a pair of stable CRTs as limiting trees.
  \item Let $n=\zeta(\ft)$. Broutin and Wang \cite{BW16} studied an \em inhomogeneous vertex splitting rule and vertex cut-tree \em ${\rm cut}_{p_n}(\ft)$ based on a
    distribution $p_n$ on vertices, and applied this to Camarri and Pitman's \cite{CP00} $p_n$-trees. They showed GP/Gromov-Hausdorff-Prokhorov (GHP)
    convergence of $p_n$-trees to Aldous and Pitman's inhomogeneous CRTs \cite{AP00} implies the convergence of pairs of trees and cut-trees in the same mode of
    convergence. This does not include conditioned Galton-Watson trees beyond a result for uniform trees of \cite{Bertoin2012}.
\end{itemize}
Before the constructions of cut-trees, the evolution of the root component had received particular attention
\cite{AD12b,ABBH14,Bertoin2012,Janson06,MeirMoon,Panholzer2006}. In the cut-tree, this \em pruning process \em corresponds to a single spine. Pruning processes of
Galton-Watson trees were studied by Aldous and Pitman \cite{AP98} under the edge splitting
rule, and by Abraham et al. \cite{ADH12} under our vertex splitting rule. Limit theorems for pruning processes were obtained in \cite{HW} in both cases. These are
for forests of Galton-Watson trees. In the domain of attraction of the Brownian forest, this is the same (up to the conditioning on numbers of leaves or vertices) as the
joint convergence of the tree and a spine of the cut-tree.

Let $\cG$ be a Galton-Watson tree. In our vertex cut-tree model, conditioning on $\lambda(\cG)=n$, the splitting rule turns out to give some random number
$k+1$ of conditioned Galton-Watson trees whose numbers of leaves add up to $n+1$. Hence, the cut-tree is almost a Markov branching tree in the sense of Haas and
Miermont \cite{HM10}. This property fails for all cut-trees of Galton-Watson trees conditioned on $\zeta(\cG)=n$, except for the edge cut-tree of a Poisson-Galton-Watson
tree, which gives the uniform model studied in \cite{AP98,Bertoin2012,Pitman1999}. Informally, the root component is biased by the number of its 
leaves. While in general, GHP convergence appears to be much harder to prove than GP convergence (hence the weaker results in
\cite{BM,Die13}), we present here a way to apply the results of \cite{HM10} and obtain the stronger mode of convergence.

One of the key ideas is not to focus on the number of leaves, but on $\overline{n}:=2n-1$. Then the ``split'' of $n$ leaves into $n_1+\cdots+n_{k+1}=n+1$
means that $\overline{n}_1+\cdots+\overline{n}_{k+1}=2(n+1)-k-1\le\overline{n}$ for all $k\ge 2$. We will obtain a Markov branching
cut-tree in terms of numbers $\overline{n}=2n-1$ associated with numbers $n$ of leaves. For $k\ge 3$, there is loss of mass, so we proceed, as follows.
\begin{itemize}\item Let $n=\lambda(\ft)$. We add $k-2$ singleton components to $\ourcut(\ft)$ for every split into $k+1$ components (summing to $2k-1=\overline{k}$ components), $k\ge 2$. We modify our
   vertex cut-tree to include the additional singleton components. We denote this vertex cut-tree by $\modcut(\ft)$.
\end{itemize}

\begin{prop}\label{propspl} Let $\cG^{(n)}$ be an $n$-leaf Galton-Watson tree with offspring distribution $\nu$. Then the vertex cut-tree $\modcut(\cG^{(n)})$ is a Markov branching tree with splitting probabilities
  $$q_{\overline{n}}(\#{\rm blocks}=\overline{k})
		=\frac{k-1}{k+1}\nu_k\frac{n+1}{n-1}\frac{\bP(S_{k+1}=n+1)}{\nu_0\bP(S_1=n)},\qquad k\ge 2,$$
  where $S_k=X_1+\cdots+X_k$ for independent ${\rm GW}(\nu)$-trees $\cG_j$ with $X_j=\lambda(\cG_j)$ leaves, $j\ge 1$; and given $\overline{k}$ blocks, the
  ranked block sizes are like the non-increasing rearrangement of $(\overline{X}_1,\ldots,\overline{X}_{k+1})$
  conditionally given $X_1+\cdots+X_{k+1}=n+1$, with an additional $k-2$ blocks of size 1 appended.
\end{prop}


We use the following notation: let
\begin{itemize}\item $\cG^{(n)}$ be a Galton-Watson tree rooted at an ancestor and conditioned to have $n$ leaves,
  \item $\ourcut(\cG^{(n)})$ our vertex cut-tree of the beginning of this introduction, where in a tree with $n$ leaves a branch point with $k$ children is
    cut with probability $(k-1)/(n-1)$,
  \item and $\modcut(\cG^{(n)})$ its modification as just above Proposition \ref{propspl}, i.e.\ $\ourcut(\cG^{(n)})$ with $k-2$ singleton blocks added to the cut-tree when cutting a
    branch point with $k$ children.
\end{itemize}
The goal is to show that suitably scaled, we get convergence to $(\cT_{\rm Br},{\rm cut}(\cT_{\rm Br}))$, where $\cT_{\rm Br}$ is the Brownian CRT and
${\rm cut}(\cT_{\rm Br})$ is the Brownian cut-tree introduced by Bertoin and Miermont \cite{BM}, see Section \ref{secbrcut}. We assume for simplicity that the offspring distribution $\nu$ satisfies $\nu_1=0$. This is no loss of generality since
our conditioning does not affect single-child vertices. To pass from this special case to the case of a general offspring distribution, we can associate the offspring
distribution conditioned not to produce a single child and represent the desired Galton-Watson tree with single-child vertices as the tree with the conditioned offspring
distribution, but with edge lengths added that are independent geometrically distributed with success parameter $1-\nu_1$.

Let us modify $\cG^{(n)}$ to a 
\begin{itemize}\item random tree $\widehat{\cG}^{(n)}$ in which every branchpoint of $\cG^{(n)}$ with $k$ children has $k-2$ more children added, who themselves have no
    offspring.
\end{itemize}
If $\cG^{(n)}$ is binary, then $\widehat{\cG}^{(n)}=\cG^{(n)}$, with $2n-1$ vertices and $2n-2$ edges. In general, the effect of this modification is that the tree
with previously $n$ leaves but fewer than $2n-2$ edges receives $k-2$ new edges for any branch point of degree $k$, for all $k\ge 2$. We note an elementary lemma.
\begin{lm}\label{edgelemma} The random tree $\widehat{\cG}^{(n)}$ has $2n-1$ vertices and $2n-2$ edges almost surely.
\end{lm}
The modification of adding $k-2$ edges to $\cG^{(n)}$ is related to adding $k-2$ singleton components to $\ourcut(\cG^{(n)})$ to form
$\modcut(\cG^{(n)})$, which we did in order to obtain a Markov branching tree without loss of mass in Proposition \ref{propspl}. In both cases the
effect on the Gromov-Hausdorff (GH) distances of the trees is an elementary consequence of the definition (recalled in Section \ref{secGHP}):
\begin{lm}\label{lmcut} We have $d_{\rm GH}(\cG^{(n)},\widehat{\cG}^{(n)})\le 1$ and $d_{\rm GH}(\ourcut(\cG^{(n)}),\modcut(\cG^{(n)}))\le 1$.
\end{lm}
After scaling, as $n\rightarrow\infty$, the GH scaling limits will be identical, i.e. the scaled pair converges to the same limiting tree.
Comparison in the GHP distance $d_{\rm GHP}$ is less straightforward.

Recall that Dieuleveut's vertex cut-tree ${\rm cut}_{\rm D}(\ft )$ selects each branch point with $k$ children with probability proportional to $k$, while our vertex
cut-tree $\modcut(\ft )$ selects each branch point with $k$ children with probability proportional to $k-1$. Now note that $\widehat{\cG}^{(n)}$ has $2k-2\ge 2$
children wherever $\cG^{(n)}$ has $k\ge 2$ children, and in $\widehat{\cG}^{(n)}$, Dieuleveut would select a branch point with $2k-2$ children with probability proportional
to $2k-2=2(k-1)$. Hence, we can couple the constructions of $\modcut(\cG^{(n)})$ and ${\rm cut}_{\rm D}(\widehat{\cG}^{(n)})$. However, Dieuleveut proceeds slightly differently when building the cut-tree. The branch points of $\modcut(\cG^{(n)})$ and ${\rm cut}_{\rm D}(\widehat{\cG}^{(n)})$ can be taken the
same, but the numbers of leaves at any particular branch point are typically different, while the total numbers of leaves are $2n-1$ and $2n-2$, respectively. See e.g. Figure \ref{fig1}. 

Specifically, for
$v\in{\rm Br}(\cG^{(n)})$ with $k=k_v(\cG^{(n)})$ children, our cut-tree $\modcut(\cG^{(n)})$ always has $k+1$ main components, some of which may be singleton vertices,
and $k-2$ more singleton components, giving $2k-1$ altogether. On the other hand, ${\rm cut}_{\rm D}(\widehat{\cG}^{(n)})$ records components of the edge set, and
depending on when the $k$ edges are removed, they may or may not have subtrees above them. As an extreme example, suppose that all $k$ of them
initially had subtrees above them, and $v$ is not the root. If this is the first split, there are $k+1$ components above and below, plus a further $k$ singletons for the
removed edges, $2k+1$ altogether. If, however, this is the last split, there are only the $k$ singletons, all other ``components'' already being empty.

\begin{prop}\label{propcut}  We have $d_{\rm GH}({\rm cut}_{\rm D}(\widehat{\cG}^{(n)}),\modcut(\cG^{(n)}))\le 1$ for a suitable coupling.
\end{prop}

Turning to $d_{\rm GHP}$, the question arises what mass measures we place onto the cut-trees. Bertoin, Miermont and Dieuleveut actually consider trees with $n$ edges ($n-1$ vertices) and obtain cut-trees with $n$ leaves, so it is natural to put the uniform measure in leaves onto their cut-trees in their framework.
In our framework, we equip $\modcut(\cG^{(n)})$ with the uniform measure on its $2n-1=\overline{n}$ leaves and ${\rm cut}_{\rm D}(\widehat{\cG}^{(n)})$ with the
uniform measure on its $2n-2$ leaves. We also equip $\cG^{(n)}$ with the uniform measure on its $n$ leaves and $\widehat{\cG}^{(n)}$ with the uniform
measure on its $2n-2$ edges. Our programme has three steps, here given for the finite variance case, for suitable $c_n$ and $c_n^\prime$, which will be discussed in Sections 
\ref{sec2.1} and \ref{sec3.1}, respectively. We show
\begin{enumerate}
  \item $\modcut(\cG^{(n)})/c_n^\prime\rightarrow\cT_{\rm Br}$ in GHP, using the Markov branching convergence criterion of \cite{HM10}, deduce $({\rm cut}_{\rm D}(\widehat{\cG}^{(n)})/c_n^\prime,\modcut(\cG^{(n)})/c_n^\prime,\ourcut(\cG^{(n)})/c_n^\prime)\rightarrow(\cT_{\rm Br},\cT_{\rm Br},\cT_{\rm Br})$ in GH$^3$;
  \item $\widehat{\cG}^{(n)}/c_n\rightarrow\cT_{\rm Br}$ in GHP, based on \cite{Mie08,deR15}, deduce $(\cG^{(n)}/c_n,\widehat{\cG}^{(n)}/c_n)\rightarrow(\cT_{\rm Br},\cT_{\rm Br})$ in GHP$^2$.
  \item $(\widehat{\cG}^{(n)}/c_n,{\rm cut}_{\rm D}(\widehat{\cG}^{(n)})/c_n^\prime)\rightarrow (\cT_{\rm Br},{\rm cut}(\cT_{\rm Br}))$, in GP$^2$, adapting the arguments of \cite{Die13}.
\end{enumerate}
Here GHP, GH$^3$, GHP$^2$ and GP$^2$ denote convergences in distribution on product spaces, where each component is equipped with the GHP, GH or GP topologies, as appropriate, see Section \ref{secGHP}. We deduce that $\cT_{\rm Br}\overset{d}{=}{\rm cut}(\cT_{\rm Br})$, as was already shown in \cite{BM}. More importantly, we conclude:
\begin{thm}\label{thm1} With any finite-variance offspring distribution $\displaystyle(\cG^{(n)}/c_n,\widehat{\cG}^{(n)}/c_n)\rightarrow(\cT_{\rm Br},\cT_{\rm Br})$ in {\rm GHP}$^2$ in distribution,
  jointly with $(\modcut(\cG^{(n)})/c_n^\prime,{\rm cut}_{\rm D}(\widehat{\cG}^{(n)})/c_n^\prime)\rightarrow({\rm cut}(\cT_{\rm Br}),{\rm cut}(\cT_{\rm Br}))$
  in {\rm GHP}$^2$, as $n\rightarrow\infty$ in $\{n\ge 1\colon\bP(\lambda(\cG)=n)>0\}$ for an associated Galton-Watson tree $\cG$.
\end{thm}
Given the three steps, the remaining proof is mainly a standard argument via tightness and uniqueness of subsequential limit distributions, see Section \ref{secpfthm}, but also requires the following result, which is part of the folklore on the Brownian CRT $(\cT_{\rm Br},\mu_{\rm Br})$, but we were unable to locate it in the literature, so we quickly derive it from well-known results in Section \ref{secpfthm}.

\begin{prop}\label{propintro} The measured tree $(\cT_{\rm Br},\mu_{\rm Br})$ is a measurable function of the unmeasured $\cT_{\rm Br}$.
\end{prop}

This proposition will also hold for stable trees, but the argument would be more involved, and since we do not need this here, we do not work out the details.

The structure of this paper is as follows. In Section \ref{sec2}, we note a local limit theorem for the number of leaves, recall the three relevant topologies GP, GH and
GHP, we prove Proposition \ref{propintro}, and we deduce Theorem \ref{thm1} from the three steps given above. In Section \ref{secmain}, we turn to the three main steps and hence complete the above programme
in the finite variance case, and we indicate how corresponding results in the stable case can be approached. Appendix \ref{appA} includes an auxiliary result to deduce joint GHP convergence from joint GP
convergence, which we do not use in this final version, but which may be of independent interest. We also include the brief Appendix \ref{appB} summarising the use of
different normalisations of the Brownian CRT in the literature.

\section{Preliminaries}\label{sec2}

\subsection{A local limit theorem for the number of leaves}\label{sec2.1}

Consider a critical offspring distribution $\nu$ in the domain of attraction of a stable distribution with index $\alpha\in(1,2]$. Specifically, suppose that for a random walk $W$ with step distribution $\bP(W_1=n)=\nu_{n+1}$, $n\ge -1$,
\begin{equation}\label{regvarass}\frac{W_n}{a_n}\convd X_1,
\end{equation}
where $a_n$ is regularly varying with index $\alpha$ and $\bE(\exp(-\lambdaLaplace X_1))=\exp(\lambdaLaplace^\alpha)$.  
Then the classical local limit theorem holds for $W$, see Ibragimov and Linnik \cite[Theorem 4.2.1]{IbL71}, or Kortchemski \cite[Theorem 1.10]{Kor12} for a statement:
$$\sup_{k\in\bZ}\left|a_n\bP(W_n=k)-p_1\left(\frac{k}{a_n}\right)\right|\rightarrow 0\qquad\mbox{as $n\rightarrow\infty$,}$$
where $p_1$ is the continuous density of $X_1$, which is $p_1(x)=\frac{1}{2\sqrt{\pi}}\exp(-x^2/4)$, $x\in\bR$, for $\alpha=2$.

Consider the stopping times $K_0=0$ and $K_{n+1}=\inf\{k\ge K_n+1\colon W_k-W_{k-1}=-1\}$ of down-moves and the time-changed process $\widetilde{W}_n=W_{K_n}$, $n\ge 0$, of values after down-moves. This can be viewed as a transformation on trees that in some sense removes all non-leaf branch points. See Rizzolo \cite{Riz13} for generalisations removing all branch points with multiplicities not in a set $A\subset\bN$. Note that the original tree can be recovered from $W$, but not in general from $\widetilde{W}$. Effectively, some of the leaves of the tree encoded in $W$ now act as branch points of the transformed tree encoded in $\widetilde{W}$ (replacing one or more removed branch points).
\begin{lm}\label{lem2.1} The increment distribution of $\widetilde{W}$ is in the domain of attraction of the same stable distribution as $\nu$. Specifically,
  $$\frac{\widetilde{W}_n}{\widetilde{a}_n}\convd X_1,$$
  where $\widetilde{a}_n=a_n/\nu_0^{1/\alpha}$. If $W_1$ has finite variance $\sigma^2$, we can choose $a_n=\sigma\sqrt{n/2}$.
\end{lm}
\begin{proof} This is rather elementary: by definition, we can write $\widetilde{W}_1=A_1+\cdots+A_G-1$, where
  $G\sim{\rm geom}(\nu_0)$ is independent of an independent and identically distributed sequence of up-moves
  $A_n$, $n\ge 1$, with $\bP(A_n=j)=\nu_{j+1}/(1-\nu_0)$, $j\ge 0$. Here
  $$\bE\left[\exp\left(-\lambdaLaplace\widetilde{W}_1\right)\right]=e^\lambdaLaplace\frac{\nu_0}{1-(1-\nu_0)\bE[e^{\lambdaLaplace A_1}]}=\frac{e^\lambdaLaplace\nu_0}{1-\left(\bE[e^{-\lambdaLaplace W_1}]-\nu_0e^\lambdaLaplace\right)}$$
  By assumption,
  $$\left(\bE\left[\exp\left(-\frac{\lambdaLaplace}{a_n}W_1\right)\right]\right)^n\longrightarrow\exp(\lambdaLaplace^\alpha)\qquad\mbox{i.e.}\qquad n\left(\bE\left[\exp\left(-\frac{\lambdaLaplace}{a_n}W_1\right)\right]-1\right)\longrightarrow\lambdaLaplace^\alpha.$$
  Hence
  $$n\left(\bE\left[\exp\left(-\frac{\lambdaLaplace}{\widetilde{a}_n}\widetilde{W}_1\right)\right]-1\right)
    =\frac{n\left(1-\bE\left[\exp\left(-\frac{\lambdaLaplace\nu_0^{1/\alpha}}{a_n}W_1\right)\right]\right)}{1-\bE\left[\exp\left(-\frac{\lambdaLaplace\nu_0^{1/\alpha}}{a_n}W_1\right)\right]+\nu_0\exp\left(\frac{\lambdaLaplace\nu_0^{1/\alpha}}{a_n}\right)}\longrightarrow\frac{(\lambdaLaplace\nu_0^{1/\alpha})^\alpha}{\nu_0}=\lambdaLaplace^\alpha.$$
  If $\sigma^2<\infty$, then $a_n=\sigma\sqrt{n/2}$ is the central limit theorem with limiting variance $2$.
\end{proof}
\begin{cor}\label{loclim} Under the assumption {\rm(\ref{regvarass})}, the time-changed process $\widetilde{W}$ satisfies the local limit theorem $$\sup_{k\in\bZ}\left|\widetilde{a}_n\bP(\widetilde{W}_n=k)-p_1\left(\frac{k}{\widetilde{a}_n}\right)\right|\rightarrow 0\qquad\mbox{as $n\rightarrow\infty$.}$$
\end{cor}
Now denote by $S_j$ respectively $S_j^V$ the random number of leaves respectively vertices in $j$ independent Galton-Watson trees with offspring distribution $\nu$. Following Haas and Miermont \cite{HM10}, we note the classical argument based on the observation that we can think of the steps of $W$ as corresponding to vertices of the trees (e.g. exploring the trees in depth first order) adding each time the number of children minus one so that $W_k$ is the number of unexplored vertices whose parent has been explored minus $j$ while the $j$th tree is being explored. Then
$$S_j^V=n\quad \iff\quad W_n=-j\ \mbox{ and }\ W_m>-j,\ m<n,$$
which via the cyclic lemma (e.g. Feller \cite[Lemma XII.6.1]{Fel2}) for the downward skip-free random walk $W$ yields\vspace{-0.2cm}
$$\bP(S_j^V=n)=\frac{j}{n}\bP(W_n=-j).$$
The following result was noted in \cite[Corollary 1]{Riz13} and has been implicit in Kortchemski \cite{Kor12}.
\begin{prop} We have $\ \displaystyle\bP(S_j=n)=\frac{j}{n}\bP(\widetilde{W}_n=-j)\ $ for all $1\le j\le n$.
\end{prop}
\begin{proof} Just note that $\widetilde{W}$ is also downward skip-free since it does not skip any down-moves of $W$. Each step now corresponds to a leaf and $-j$ is first reached when all leaves have been explored so that\vspace{-0.2cm}
$$S_j=n\quad \iff\quad \widetilde{W}_n=-j\mbox{ and }\widetilde{W}_m>-j,m<n,$$
and we conclude via the cyclic lemma for $\widetilde{W}$.
\end{proof}
\begin{cor} We have $\ \displaystyle\sup_{j\ge 1}\left|n\widetilde{a}_n\frac{1}{j}\bP(S_j=n)-p_1\left(\frac{-j}{\widetilde{a}_n}\right)\right|\rightarrow 0\ $ as $n\rightarrow\infty$.
\end{cor}

%

Recall that given a planar tree ${\bf t}$ with root $\rho$, we denote by $\zeta({\bf t})$ and $\lambda({\bf t})$ the total number of vertices and leaves of ${\bf t}$, 
respectively. For $v\in V(\ft)$ with $v=v_k\rightarrow v_{k-1}\rightarrow\cdots\rightarrow v_1\rightarrow v_0=\rho$, we say that $v$ has \em generation \em $|v|=k$. 
Denote by $\zeta_k({\bf t})$ and $\lambda_k({\bf t})$ the number of vertices and leaves of ${\bf t}$ at generation $k$. Let ${\bf t}(k)$ be ${\bf t}$ restricted to
generation at most $k$, i.e.
$$
{\bf t}(k)=\{v\in {\bf t}: |v|\leq k\}.
$$
Let ${\cal G}^{(n)}$ be a critical Galton-Watson tree conditioned to have $n$ leaves, with offspring distribution $\nu$, and $\widehat{\cG}^{(n)}$ its modification with 
extra leaves as  defined just before Lemma \ref{edgelemma}.
\begin{lm}\label{bound} If the offspring distribution has finite variance, there exists a constant $C>0$ such that
$$\sup_{n\geq1}\bE\left[ \zeta_k\left(\widehat{{\mathcal G}}^{(n)}\right)\right]\leq 2\sup_{n\geq1}\bE\left[ \zeta_k\left({\mathcal G}^{(n)}\right)\right]\leq C k,\quad k\geq1.$$
\end{lm}
\begin{proof}
We adapt Janson's idea of proving \cite[Theorem 1.13]{Janson06}. Our proof will be divided into four subparts. We use $c,C, C_1, C_2, \ldots$ for constants independent of
$n$ and $k$.

\medskip

\noindent {\it Subpart 1.}
 Let $\cG$ be a Galton-Watson tree and ${\cal G}^{\infty}$ the so-called Kesten tree arising as local limit of ${\cal G}^{(n)}$ as $n\rightarrow\infty$; see Abraham and Delmas \cite{AD14}. It is well-known \cite[(1.15)]{Kes86} that for any tree $\bf t$\vspace{-0.2cm}
$$
\bP({\cal G}(k)= \ft(k))=\zeta_k(\ft)\bP({\cal G}^{\infty}(k)=\ft(k)).
$$
Let $\ft$ be a tree with $\zeta_k(\ft)=m$. Define $N=n-\sum_{i\leq k-1}\lambda_i(\ft(k)).$ Then
by conditioning on generation $k$ and using Kortchemski \cite[Theorem 3.1]{Kor12} and Proposition 2.3, we obtain
\begin{eqnarray}
\bP({\cal G}^{(n)}(k)= \ft(k))
&=& \frac{\bP({\cal G}(k)=\ft(k),\lambda({\cal G})=n)}{\bP(\lambda({\cal G})=n)}\nonumber\\
&\leq & C_1n^{3/2} \bP({\cal G}(k)=\ft(k))\bP\left(S_m=N\right)\nonumber\\
&=&C_1n^{3/2} \bP({\cal G}(k)=\ft(k))\frac{m}{N}\bP\left(\widetilde{W}_N=-m\right)\nonumber\\
&\leq&C_2m\left(\frac{n}{N}\right)^{3/2} e^{-cm^2/N}\bP({\cal G}(k)=\ft(k))\nonumber\\
&=&C_2\left(\frac{n}{N}\right)^{3/2} e^{-cm^2/N}\bP({\cal G}^{\infty}(k)=\ft(k)),\label{RNderivative}
\end{eqnarray}
where in the second inequality we use  Lemma  2.1 above and \cite[Lemma 2.1]{Janson06}.

\medskip
The argument in Subparts 2.--\,4. is very similar to the proof of \cite[Theorem 1.13]{Janson06} with only slight modifications.

\medskip

\noindent{\it Subpart 2.}
For each $k\geq1$, define
$$
\Gamma_k=\left\{\sum_{i\le k-1}\lambda_i({\cal G}^{(n)}(k))\leq n/2 \right\}\qquad\mbox{and}\qquad
{\zeta}_k^*({\cal G}^{(n)}(k))={\zeta}_k({\cal G}^{(n)}(k))1_{\Gamma_k}.
$$
By (\ref{RNderivative}),  for any tree $\ft$ with $\sum_{i\le k-1}\lambda_i(\ft(k))\leq n/2$ and $ \zeta_k(\ft)>0$,
we have $$\bP({\cal G}^{(n)}(k)= \ft(k))
\leq C_3\bP({\cal G}^{\infty}(k)=\ft(k)),$$
which implies
$$
\bP({\zeta}_k^*({\cal G}^{(n)})=i)
\leq C_4\bP(\zeta_k({\cal G}^{\infty})=i),\qquad \mbox{for all }i\geq1.
$$
Thus
\begin{eqnarray}\label{bounda}
\bE[ {\zeta}_k^*({\cal G}^{(n)})]=\bE[ \zeta_k({\cal G}^{(n)})1_{\Gamma_k}]\leq C_4 \bE[ \zeta_k({\cal G}^{\infty})]\leq C_5k,
\end{eqnarray}
where the last inequality follows from \cite[Lemma 2.3]{Janson06}.

\medskip

\noindent{\it Subpart 3.} On $\Gamma_k^c$, one can find a (random) integer $L\leq k$ such that
$$
\sum_{i=1}^{L-1}{\lambda}_i({\cal G}^{(n)})\leq n/2\qquad\mbox{and}\qquad \sum_{i=1}^{L}{\lambda}_i({\cal G}^{(n)})> n/2.
$$
Thus on $\Gamma_k^c$,
$$
\sum_{i=0}^{k}{\zeta}_i^*({\cal G}^{(n)})=\sum_{i=0}^{L}{\zeta}_i^*({\cal G}^{(n)})=\sum_{i=0}^{L}{\zeta}_i({\cal G}^{(n)})>\sum_{i=0}^{L}{\lambda}_i({\cal G}^{(n)})> n/2.
$$
By the Markov inequality and (\ref{bounda}),
$$
\bP(\Gamma_k^c)\leq  \frac{2}{n}\bE\left[\sum_{i=0}^{k}{\zeta}_i^*({\cal G}^{(n)}) \right]\leq \frac{C_6k^2}{n}.
$$
Hence, we obtain
\begin{eqnarray}\label{boundb}
\bE\left[ \zeta_k({\cal G}^{(n)})1_{\Gamma_k^c} 1_{\{ \zeta_k({\cal G}^{(n)})\leq\sqrt{n}\}}\right]\leq \sqrt{n}\bP(\Gamma_k^c)\leq \sqrt{n \bP(\Gamma_k^c)} \leq \sqrt{C_6}k.
\end{eqnarray}

\medskip

\noindent{\it Subpart 4.}  For any $\ft$ with $\zeta_k(\ft)\geq \sqrt{n}$, according to (\ref{RNderivative}), we have
$$\bP({\cal G}^{(n)}(k)= \ft(k))
\leq  C_7\left(\frac{n}{N}\right)^{3/2} e^{-cn/N}\bP({\cal G}^{\infty}(k)=\ft(k))\leq C_8 \bP({\cal G}^{\infty}(k)=\ft(k)),$$
which, by reasoning similar as for (\ref{bounda}),  yields
\begin{eqnarray}\label{boundc}
\bE[ \zeta_k({\cal G}^{(n)})1_{\{\zeta_k({\cal G}^{(n)})> \sqrt{n}\}}]\leq C_2 \bE[ \zeta_k({\cal G}^{\infty})]\leq C_9k.
\end{eqnarray}
Then the desired result follows from (\ref{bounda}), (\ref{boundb}) and (\ref{boundc}). We have completed the proof.\end{proof}

\subsection{GH, GP and GHP topologies}\label{secGHP}

According to \cite{EPW06,EW06,GPW,Mie09} and references therein, we can define a
Gromov-Hausdorff-Prokhorov (Gromov-Hausdorff or Gromov-Prokhorov) distance on the set of measure-preserving isometry classes of pointed measured compact metric
spaces to turn the set (of equivalence classes modulo measure or modulo restriction to the support of the measure) into a Polish space.

Specifically, let $(Z,d^Z)$ be a metric space. For Borel sets $A,B\subseteq Z$, set
$$
 d_\text{H}^Z(A,B)= \inf \{ \varepsilon >0\colon A\subseteq B^\varepsilon\
 \mathrm{and}\ B\subseteq
A^\varepsilon \},
$$
the Hausdorff distance between $A$ and $B$, where $A^\varepsilon = \{ x\in Z\colon
\inf_{y\in A} d^Z(x,y) \le \varepsilon\}$.
Let $M_{f}(Z)$ be the set of all Borel probability measures  on $(Z,d^Z)$.
For $\mu,\mu^\prime \in M_f(Z)$, we define
$$
 d_\text{P}^Z(\mu,\mu^\prime) = \inf \{ \varepsilon >0\colon \mu(A)\le \mu^\prime(A^\varepsilon) +
 \varepsilon
\text{ and }
\mu^\prime(A)\le \mu(A^\varepsilon)+\varepsilon\
\text{ for all closed } A\subseteq Z \},
$$
the Prokhorov distance between $\mu$ and $\mu^\prime$.

A pointed measured  metric space ${\rm T} = (T, d,  \rho,\mu)$ is a
metric space $(T, d)$ with a distinguished element $\rho\in T$
and  a Borel probability measure  $\mu$ on $(T,d)$.
For two compact pointed measured metric spaces ${\rm T}=(T,d,\rho,\mu)$ and ${\rm T}'=(T',d',\rho',\mu')$, the Gromov-Hausdorff-Prokhorov  distance is\vspace{-0.2cm}
$$
 d_{\text{GHP}}({\rm T},{\rm T}') = \inf_{\Phi,\Phi',Z} \left(
 d_\text{H}^Z(\Phi(T),\Phi'(T')) +
d^Z(\Phi(\rho),\Phi'(\rho')) + d_\text{P}^Z(\Phi_* \mu,\Phi_*'
\mu') \right),\vspace{-0.1cm}
$$
where the infimum is taken over all isometric embeddings $\Phi\colon T\hookrightarrow
Z$ and $\Phi'\colon T'\hookrightarrow Z$ into some common Polish metric space
$(Z,d^Z)$ and $\Phi_* \mu$ is the measure $\mu$ transported by $\Phi$.
Similarly, we define Gromov-Hausdorff and Gromov-Prokhorov distances, respectively, as
$$
 d_{\text{GH}}({\rm T},{\rm T}') = \inf_{\Phi,\Phi',Z} \left(
 d_\text{H}^Z(\Phi(T),\Phi'(T')) +
d^Z(\Phi(\rho),\Phi'(\rho'))  \right),
$$
$$
 d_{\text{GP}}({\rm T},{\rm T}') = \inf_{\Phi,\Phi',Z} \left(
d^Z(\Phi(\rho),\Phi'(\rho')) + d_\text{P}^Z(\Phi_* \mu,\Phi_*'
\mu') \right).
$$
A compact metric space $(T,d)$ is called a real tree if for any two $x,y\in T$, there is an isometry $f_{x,y}\colon[0,d(x,y)]\rightarrow T$ with $f_{x,y}(0)=x$ and 
$f_{x,y}(d(x,y))=y$, and if for all injective $g\colon[0,1]\rightarrow T$ with $g(0)=x$ and $g(1)=y$ we have $g([0,1])=f_{x,y}([0,d(x,y)])$. Every real tree is naturally
equipped with a (sigma-finite) length measure $\ell$, for which $\ell(f_{x,y}([0,d(x,y)]))=d(x,y)$, $x,y\in T$. We refer to a pointed real tree $(T,d,\rho)$ as a \em rooted \em real tree, to points $x\in T\setminus\{\rho\}$ for which $T\setminus\{x\}$ is connected, respectively, disconnected into three or more connected components, 
as \em leaves\em, respectively \em branch points\em.

For any rooted real tree $(T,d,\rho)$, we define the height ${\rm ht}(T)=\max\{d(\rho,x),x\in T\}$. For any $x\in T$, we define the subtree 
$T_x=\{y\in T\colon x\in f_{\rho,y}([0,d(\rho,y)])\}$ above $x$. For $\varepsilon>0$, we define Neveu's \cite{Nev86b} notion of $\varepsilon$-erasure of $T$ as $R_\varepsilon(T)=\{\rho\}\cup\{x\in T\colon{\rm ht}(T_x)\ge\varepsilon\}$. Then $R_\varepsilon(T)$ is a rooted real tree with finitely many leaves and branch points; see also \cite{EPW06,NeP89b,NeP89a}. 

Examples of pointed measured compact real trees are obtained from continuous functions $h\colon[0,1]\rightarrow[0,\infty)$. For $s,t\in[0,1]$, let 
$d_h(s,t)=h(t)+h(s)-2\inf\{h(r),\min(s,t)\!\le\! r\!\le\!\max(s,t)\}$ and $s\sim_h t$ iff $d_h(s,t)=0$. Then the quotient space $T_h=[0,1]/\sim_h$ is a compact real tree when equipped 
with the quotient metric, again denoted by $d_h$. We further equip $(T_h,d_h)$ with the root $\rho_h=[0]_{\sim_h}$ and the measure $\mu_h$ obtained as the  
push-forward of Lebesgue measure on $[0,1]$ under the quotient map. The function $h$ is called the \em height function \em of $(T_h,d_h,\rho_h,\mu_h)$.

\subsection{Bertoin and Miermont's Brownian cut-tree}\label{secbrcut}

A Brownian Continuum Random Tree (CRT) is a random pointed measured compact metric space introduced by Aldous \cite{Ald91}. One construction is to take $h=2B^{\rm ex}$ 
as height function, for a normalised excursion $B^{\rm ex}$ of linear Brownian motion. 

Let $({\cal T}_{\rm Br},\mu_{\rm Br})$ be a Brownian CRT. Conditionally on ${\cal T}_{\rm Br}$, let $\sum_{i\in I}\delta_{(t_i, x_i)}(dt, dx)$ be a Poisson point measure on $[0, \infty)\times {\cal T}_{\rm Br}$ with intensity $dt\times d{\ell}_{\rm Br}$, where $\ell_{\rm Br}$ is the length measure on ${\cal T}_{\rm Br}$. Denote by ${\cal T}_{\rm Br}(t)$ the ``forest" obtained by removing points $\{x_i\colon i\in I, t_i\leq t\}$ that are marked before $t$. For any $x\in {\cal T}_{\rm Br}$,  let ${\cal T}_{\rm Br}(x, t)$ be the connected component of ${\cal T}_{\rm Br}(t)$ that contains $x$ with the convention that ${\cal T}_{\rm Br}(x, t)=\emptyset$ if $x\notin {\cal T}_{\rm Br}(t)$. Define ${\mu}_{\rm Br}(x, t)=\mu_{\rm Br}({\cal T}_{\rm Br}(x, t) )$.  We further define a function $\delta $ from $({\cal T}_{\rm Br}\cup \{0\})^2$ into $[0, +\infty]$ such that $\delta(0,0)=0$ and 
$$\delta(0, x)=\delta(x, 0)=\int_0^{\infty} {\mu}_{\rm Br}(x, t) dt\qquad\mbox{and}\qquad
\delta(x, y)=\int_{t(x, y)}^{\infty}\left( {\mu}_{\rm Br}(x, t)+{\mu}_{\rm Br}(y, t)\right)dt,$$
where $t(x, y)=\inf\{t\geq0\colon {\cal T}_{\rm Br}(x, t)\neq{\cal T}_{\rm Br}(y, t)\}.$

Let $\xi_0=0$ and $(\xi_i, i\in \bN)$ be an i.i.d. sequence distributed as $\mu_{\rm Br}$. For all $k\ge 1$, let ${\cal R}_k$ be the random real tree spanned by $\{\xi_0, \xi_1,\ldots, \xi_k\}$ and $\delta$.  Then $\text{cut}({\cal T}_{\rm Br})$ is defined as 
$$
\text{cut}({\cal T}_{\rm Br})=\overline{\bigcup_{k\geq 1}{\cal R}_k}\,,
$$
the completion of the metric space $(\bigcup_{k\geq1}{\cal R}_k, \delta)$. Then $({\rm cut}(\cT_{\rm Br}),\delta,0)$, equipped 
with the limiting empirical measure of $(\xi_i,i\in\bN)$, is again a Brownian CRT; see Bertoin and Miermont \cite{BM}.

\subsection{Deduction of Theorem \ref{thm1} from the statements of the three steps.}\label{secpfthm}

Since the proof of Theorem \ref{thm1} requires Proposition \ref{propintro}, we prove the proposition first.

\begin{proof}[Proof of Proposition \ref{propintro}] First consider $H=2B$ for a Brownian motion $B$. For $\varepsilon>0$, we follow 
  \cite[Section 7.6]{csp} and define alternating up- and down-crossing times as $D_0^{(\varepsilon)}=0$ and, for $m\ge 0$, 
  \begin{eqnarray*}U_{m+1}^{(\varepsilon)}&=&\inf\{t\ge D_m^{(\varepsilon)}\colon H(t)-\min\{H(s),D_m^{(\varepsilon)}\le s\le t\}=\varepsilon\},\\
                D_{m+1}^{(\varepsilon)}&=&\inf\{t\ge U_{m+1}^{(\varepsilon)}\colon H(t)-\max\{H(s),U_{m+1}^{(\varepsilon)}\le s\le t\}=-\varepsilon\}.
  \end{eqnarray*}
  Then $D_m^{(\varepsilon)}$ is precisely $\varepsilon$ below a previous local maximum of $H$ for all $m\ge 1$. Let 
  $X_m^{(\varepsilon)}=H(D_m^{(\varepsilon)})$ and $Y_m^{(\varepsilon)}=\min\{H(s)\colon D_{m}^{(\varepsilon)}\le s\le D_{m+1}^{(\varepsilon)}\}$, $m\ge 0$. 
  
  The excursions above the minimum of $H$ are scaled copies of $2B^{\rm ex}$ and hence encode scaled Brownian CRTs. The subtrees spanned by 
  $D_m^{(\varepsilon)}$, $m\ge 1$, are $\varepsilon$-erasures of the Brownian CRTs with leaves at heights $X_m^{(\varepsilon)}-\min\{Y_k^{(\varepsilon)},0\le k\le m-1\}$,
  $m\ge 1$, and roots and branch points at heights $Y_m^{(\varepsilon)}$, $m\ge 0$. Consider the function $H^{(\varepsilon)}$, which is piecewise linear at alternating
  slopes of $\pm 2/\varepsilon$ interpolating the alternating walk $X_0^{(\varepsilon)},Y_0^{(\varepsilon)},X_1^{(\varepsilon)},Y_1^{(\varepsilon)},\ldots$. By 
  \cite[Corollary 7.17]{csp}, we have $H^{(\varepsilon)}\rightarrow H$ locally uniformly and almost surely, as $\varepsilon\downarrow 0$.

  Our aim is to deduce that the $\varepsilon$-erasure $R_\varepsilon(\cT_{\rm Br})$ equipped with a scaled length measure 
  $\mu_\varepsilon=\varepsilon\ell_{\rm Br}|_{R_\varepsilon(\cT_{\rm Br})}$ converges to $(\cT_{\rm Br},\mu_{\rm Br})$ in GHP.

  The convergence $H^{(\varepsilon)}\rightarrow H$ includes the height function of the first excursion of height greater than $r>0$, jointly with the excursion length,
  so that convergence holds under the Brownian It\^o excursion measure $n_{\rm Br}$ conditioned on excursions of height greater than $r$, for all $r>0$. See
  \cite[Chapter XII]{RY}. 
  By disintegration of $n_{\rm Br}$ (e.g. \cite[Theorem 22.15]{kal}), this convergence also holds under the distribution of $2B^{\rm ex}$, which is
  the normalised excursion measure $n_{\rm Br}(\,\cdot\;|\,\zeta=1)$, where $\zeta(h)=\inf\{t\ge 0\colon h(t)=0\}$, for continuous $h\colon[0,1]\rightarrow[0,\infty)$.
  
  $h^{(\varepsilon)}$ pushes forward Lebesgue measure onto $\varepsilon$ times the length measure of $R_\varepsilon(T_h)$.
  Uniform convergence jointly with excursion lengths implies GHP convergence of encoded trees equipped with the push-forward of Lebesgue measure (see e.g. \cite{ADHo}). 
  This completes the proof.
\end{proof}

We noted in the introduction that while Proposition \ref{propintro} will also hold for stable trees, the argument will be more involved and beyond the scope of this
paper, since we focus on the Brownian case here. While $\varepsilon$-erasure of stable trees has been studied in \cite{DuWi12}, this paper does not construct the
mass measure from the length measure. \cite{DuLG} study height functions, but ``Poisson sampling'' instead of $\varepsilon$-erasure. For Poisson sampling, their results yield the
analogous almost sure and locally uniform convergence of contour functions. While \cite{DuWi12} have shown that $\varepsilon$-erasure and Poisson sampling yield the same marginal
distribution, the joint distributions are not the same, and hence we only obtain convergence in distribution. But this is not good enough here. To study $\varepsilon$-erasure
directly and get almost sure convergence in GHP back to the stable tree, \cite{DuWi13} may help, where a reconstruction procedure demonstrates how subtrees (which contain
all the mass) are attached to the $\varepsilon$-erased tree in order to get the stable tree back.

\begin{proof}[Proof of Theorem \ref{thm1}] From the three steps listed in the introduction (and completed in Section \ref{secmain}), we have marginal convergence in GH or GHP for each of the four components of $(\cG^{(n)}/c_n,\widehat{\cG}^{(n)}/c_n,{\rm cut}_{\rm HW}(\cG^{(n)}/c_n^\prime,{\rm cut}_{\rm D}(\widehat{\cG}^{(n)}/c_n^\prime))$. As GH-tightness implies GHP-tightness (see Miermont
  \cite[Proposition 8]{Mie09}), the joint laws are GHP$^4$-tight. Take any subsequence along which we have convergence in distribution in GHP$^4$. By
  Skorokhod's representation theorem, we may assume that convergence holds almost surely, to a vector $((\cT_1,\mu_1),\ldots,(\cT_4,\mu_4))$ of measured limiting trees.
  
  As GHP$^2$-convergence implies GP$^2$-convergence, we get from Step 3.
  $((\cT_2,\mu_2),(\cT_4,\mu_4))\sim((\cT_{\rm Br},\mu_{\rm Br}),({\rm cut}(\cT_{\rm Br}),\mu_{\rm cut}))$, by uniqueness of GP$^2$-limits. By Step 2., we obtain
  $(\cT_1,\mu_1)=(\cT_2,\mu_2)$ a.s.. By Step 1., we obtain $\cT_3=\cT_4$ a.s.. Finally, $(\cT_{\rm Br},\mu_{\rm Br})$ is a measurable function of $\cT_{\rm Br}$, by Proposition \ref{propintro}, and therefore, both $(\cT_3,\mu_3)$ by GP convergence in Step 3. and $(\cT_4,\mu_4)$ by GHP convergence in Step 1. are this measureable function of
  $\cT_3=\cT_4$ a.s.. This completely specifies the joint distribution of $((\cT_1,\mu_1),\ldots,(\cT_4,\mu_4))$, which furthermore does not depend on the chosen subsequence.
  Therefore, joint convergence in distribution holds with the limiting distribution thus identified.
\end{proof}

\section{Proof of the statements of the three steps }\label{secmain}

\subsection{Step 1: GHP convergence of vertex cut-trees as Markov branching trees}\label{sec3.1}

\begin{proof}[Proof of Proposition \ref{propspl}] Denote by $\bT$ the set of (combinatorial) rooted planar trees. Let $\cG$ be a $\bT$-valued Galton-Watson tree,
  and denote by $X=\lambda(\cG)$ the number of leaves of $\cG$. First
  note that for all trees $\ft \in\bT$ with $n$ leaves and root $\rho$, we have
  $$\bP(\cG=\ft\,|\,\lambda(\cG)=n)=\bP(\cG=\ft\,|\,X=n)=\frac{\bP(\cG=\ft )}{\bP(X=n)}=\frac{1}{\bP(X=n)}\prod_{v\in V(\ft)}\nu_{k_v(\ft )},$$
  where $V(\ft )$ denotes the set of vertices of $\ft $ and $k_v(\ft)$ the degree (number of subtrees of vertex $v\in V(\ft)$, not counting the component containing $\rho$).
  For any branch point $v\in{\rm Br}(\ft)$, splitting $\ft$
  into $\ft_1,\ldots,\ft_{k+1}$ by removing the edges $w\rightarrow v$ for all children $w$ of $v$,
  where $\ft_1$ is the component containing $\rho$ and $v$, and $\ft_2,\ldots,\ft_{k+1}$ are the components of each of the children of $v$, in planar order, we obtain
  $$\prod_{j=1}^{k+1}\bP(\cG=\ft_j)=\frac{\nu_0}{\nu_k}\bP(\cG=\ft ).$$
  Note that if we also record the new leaf $v\in{\rm Lf}(\ft_1)$, we can uniquely reconstruct $(\ft ,v)$ from $(\ft_1,\ldots,\ft_{k+1},v)$. Hence, the probability
  that the first cut is at a branch point with $k$ children is
  \begin{eqnarray*}q_{\overline{n}}(\#{\rm blocks}=\overline{k})&:=&\sum_{\ft \in\bT\colon\lambda(\ft )=n}\,\sum_{v\in{\rm Br}(\ft )\colon k_v(\ft )=k}\bP(\cG=\ft\,|\,\lambda(\cG)=n)\ \frac{k-1}{n-1}\\
    &=&\sum_{\genfrac{}{}{0pt}{}{\ft_1,\ldots,\ft_{k+1}\in\bT\colon}{\lambda(\ft_1)+\cdots+\lambda(\ft_{k+1})=n+1}}\sum_{v\in{\rm Lf}(\ft _1)}\frac{1}{\bP(X=n)}\ \ \frac{k-1}{n-1}\ \ \frac{\nu_k}{\nu_0}\ \prod_{j=1}^{k+1}\bP(\cG=\ft_j).
  \end{eqnarray*}
  By symmetry, this value is exactly the same if the second sum is taken over $v\in{\rm Lf}(\ft_i)$ for any $i=1,\ldots,k+1$. Hence, summing over $i$ and dividing by $k+1$, the second sum captures all $n+1$ leaves leaving the first sum to sum over all $(k+1)$-tuples of trees with total $n+1$ leaves, so that
  $$q_{\overline{n}}(\#{\rm blocks}=\overline{k})=\frac{(n+1)(k-1)\nu_k\bP(X_1+\cdots+X_{k+1}=n+1)}{(k+1)\bP(X=n)(n-1)\nu_0},$$
  for independent $X_j$, $1\le j\le k+1$, with the same distribution as $X=\lambda(\cG)$, as required. The joint distribution of the $k+1$ non-trivial and $k-2$ trivial components follows by a refinement of the above argument: denote by $\cS_1$ the root
  component and by $\cS_2,\ldots,\cS_{k+1}$ the subtrees, then the argument yields a probability to see a Galton-Watson tree $\cG^{(n)}$ with $n$ leaves split into $\cS_1=\fs_1$
  and $\cS_2=\fs_2,\ldots,\cS_{k+1}=\fs_{k+1}$ of
  $$\frac{1}{\bP(X=n)}\frac{k-1}{n-1}\frac{\nu_k}{\nu_0}\lambda(\fs_1)\bP(\cG=\fs_1)\prod_{j=2}^{k+1}\bP(\cG=\fs_j),$$
  and a simple combinatorial argument to handle equal block sizes yields the probability that the ranked split of $n+1$ is $(\lambda(\cS_1),\ldots,\lambda(\cS_{k+1}))^\downarrow=(m_1,\ldots,m_{k+1})$ as
  $$\frac{1}{\bP(X=n)}\frac{k-1}{n-1}\frac{\nu_k}{\nu_0}\frac{(n+1)k!}{\prod_{1\le\ell\le n}r_\ell!}\prod_{j=1}^{k+1}\bP(X=m_j),$$
  where $r_\ell=\#\{1\le j\le k+1\colon m_j=\ell\}$ is the number of block sizes equal to $\ell$. Hence, the conditional probability to see a split into
  $\cS_1=\fs_1,\ldots,\cS_{k+1}=\fs_{k+1}$ given a ranked split of $(m_1,\ldots,m_{k+1})$ is
  $$\frac{\lambda(\fs_1)}{n+1}\frac{\prod_{1\le\ell\le n}r_\ell!}{k!}\prod_{j=1}^{k+1}\bP(\cG=\fs_j\,|\,X=\lambda(\fs_j)).$$
  The Markov branching property follows if we can show that conditionally given the ranked split $(m_1,\ldots,m_{k+1})$, the multiset of trees $\{\!\{\cS_1,\ldots,\cS_{k+1}\}\!\}$
  has the same distribution as the multiset of $k+1$ independent trees with respective distribution $\bP(\cG=\cdot\,|\,X=m_j)$, $1\le j\le k+1$. First suppose that the trees
  $\ft_1,\ldots,\ft_{k+1}$ are distinct. Then the probability that the multiset of trees $\{\!\{\cS_1,\ldots,\cS_{k+1}\}\!\}$ equals $\{\ft_1,\ldots,\ft_{k+1}\}$ is the sum over all
  $\fs_1,\ldots,\fs_{k+1}$ that are permutations of $\ft_1,\ldots,\ft_{k+1}$. In particular, $\fs_1$ can be any $\ft_i$, giving different factors $\lambda(\ft_i)$, and there are $k!$
  equally likely ways to match the others:
  $$k!\sum_{i=1}^{k+1}\frac{\lambda(\ft_i)}{n+1}\frac{\prod_{1\le\ell\le n}r_\ell!}{k!}\prod_{j=1}^{k+1}\bP(\cG=\fs_j\,|\,X=\lambda(\fs_j))=\prod_{1\le\ell\le n}r_\ell!\prod_{j=1}^{k+1}\bP(\cG=\fs_j\,|\,X=\lambda(\fs_j)).$$
  When some of the trees $\ft_1,\ldots,\ft_{k+1}$ are equal, there is duplication in some of the matchings of $\fs_1,\ldots,\fs_{k+1}$ and $\ft_1,\ldots,\ft_{k+1}$, and we lose some factors from $\prod_{1\le\ell\le n}r_\ell!$.
  In each case, we get the probability that the multiset of independent conditioned Galton-Watson trees equals the multiset of $\ft_1,\ldots,\ft_{k+1}$, as required.
\end{proof}

\begin{prop}\label{cutmb} Suppose $\alpha=2$ and the offspring variance $\sigma^2$ is finite. Let $(T_n,n\ge 1)$ be a family of Markov branching trees with splitting rule as given in Proposition \ref{propspl}, so that $T_n$ is the genealogical tree of a fragmentation process starting from an initial block of size $\overline{n}=2n-1$,
equipped with the uniform measure on the $\overline{n}$ leaves of $T_n$. Then
  $$\frac{T_n}{\sqrt{n}}\rightarrow
                                    \frac{\sqrt{\nu_0}}{\sigma}\cT_{\rm Br},\qquad\mbox{in distribution in GHP},$$
where $\cT_{\rm Br}$ is a Brownian Continuum Random Tree equipped with its usual mass measure.
\end{prop}
\begin{proof} Like Rizzolo \cite{Riz13} who applied the arguments of \cite[Section 5.1]{HM10} for his results on trees with numbers of vertices in a given set of degrees, we only present the part of the argument that differs from their's in some details and thereby reveals the constants in the limiting expression. Let $\ell_1([0,\infty))$ be the space of nonnegative summable sequences with sum bounded by 1 equipped with the $\ell_1$-norm, and $f\colon\ell_1([0,\infty))\rightarrow[0,\infty)$ bounded continuous. Set $g(x)=(1-\max x)f(x)$. Then numerous applications of the local limit theorem (Corollary \ref{loclim}) yield that for all $\eta>0$ and $\eta^\prime<\eta$ small enough there is $n_0\ge 1$ and $\varepsilon>0$ such that for all $n\ge n_0$, $1\le k\le\varepsilon\sqrt{n}$ and $m=(m_1,\ldots,m_{\overline{k}})$ with $n^{1/8}\le m_1\le(1-\eta)n$ and $(1-\eta^\prime)n\le m_1+m_2\le n$ and $\overline{m}_1+\cdots+\overline{m}_{\overline{k}}=\overline{n}$
\begin{eqnarray*}\frac{(k-1)\nu_k}{\nu_0}(1-\eta)\le &q_{\overline{n}}(\#{\rm blocks}=\overline{k})&\le\frac{(k-1)\nu_k}{\nu_0}(1+\eta)\\
\left(g\left(\frac{(\overline{m}_1,\overline{n\!-\!m_1},1,0,\ldots)}{\overline{n}}\right)\!-\!\eta\right)^+\le &g\left(\frac{\overline{m}}{\overline{n}}\right)& \le g\left(\frac{(\overline{m}_1,\overline{n\!-\!m_1},1,0,\ldots)}{\overline{n}}\right)\!+\!\eta\\
\frac{k+1}{(n+1)^{3/2}}\frac{\sqrt{\nu_0}}{\sqrt{2\pi\sigma^2}}(1-\eta)\le &\bP(\tau_{k+1}=n+1)&\le\frac{k+1}{(n+1)^{3/2}}\frac{\sqrt{\nu_0}}{\sqrt{2\pi\sigma^2}}(1+\eta)\\
\frac{1}{m_1^{3/2}}\frac{1}{m_2^{3/2}}\frac{\nu_0}{2\pi\sigma^2}(1-\eta)^2\le &\!\!\!\!\bP(X_1=m_1)\bP(X_2=m_2)\!\!\!\!&\le
\frac{1}{m_1^{3/2}}\frac{1}{m_2^{3/2}}\frac{\nu_0}{2\pi\sigma^2}(1+\eta)^2\\
1-\eta\le&\bP(\tau_{\varepsilon\sqrt{n}}\le\eta^\prime n)&\le 1\\[-0.1cm]
\left(1-\frac{\eta^\prime}{\eta}\right)(n-m_1)\le &m_2&\le n-m_1.
\end{eqnarray*}
By taking $\limsup$ and $\liminf$ as $n\rightarrow\infty$ and then the limit as $\eta\rightarrow 0$, under which contributions outside the above ranges of $k$, $m_1$ and $m_2$ vanish, we see that
\begin{eqnarray*}\sqrt{n}\overline{q}_{\overline{n}}^*(g)
&\!\!\sim\!\!&
\sqrt{n}\sum_{k\ge 1}q_{\overline{n}}(\#{\rm blocks}=k)
\sum_m g\left(\frac{\overline{m}}{\overline{n}}\right)
\frac{(k+1)\overline{m}_1}{\overline{n}}\frac{k\overline{m}_2}{\overline{n}-\overline{m}_1}
\frac{\bP(X_1=m_1)\bP(X_2=m_2)}{\bP(\tau_{k+1}=n+1)}\\[-0.1cm]
&&\hspace{3cm}\bP(X_3^*=m_3,\ldots,X_{\overline{k}}^*=m_{\overline{k}}|X_1^*=m_1,X_2^*=m_2,\tau_{k+1}=n+1)
\\
&\!\!\longrightarrow\!\!&\frac{1}{\sqrt{2\pi\sigma^2\nu_0}}\sum_{k\ge 1}(k-1)k\nu_k\int_0^1g(x,1-x,0,\ldots)\frac{1}{x^{1/2}(1-x)^{3/2}}dx,
\end{eqnarray*}
where the first line only fails to be an equality because $X^*=(X_1^*,\ldots,X_{\overline{k}}^*)$ is a size-biased rearrangement of $(X_1,\ldots,X_{k+1},1,\ldots,1)$, so the exact expressions in the negligible cases where $m_1=1$ or $m_2=1$ are different. Since $\sum_{k\ge 1}(k-1)k\nu_k=\sigma^2$, we conclude by the convergence theorem of Haas and Miermont, \cite[Theorem 1]{HM10}. In particular we see from the multiplicative constant of the limiting measure that the limiting tree has
(ranked) dislocation measure 
$$\frac{\sigma}{\sqrt{2\pi\nu_0}}\left(\frac{1}{x^{1/2}(1-x)^{3/2}}+\frac{1}{(1-x)^{1/2}x^{3/2}}\right)1_{[1/2,1)}(x)dx=\frac{\sigma}{2\sqrt{\pi\nu_0}}\nu_B(dx),$$
which is associated with $\sqrt{\nu_0}\sigma^{-1}\cT_{\rm Br}$; see Appendix \ref{appB} for a discussion of normalisations of the Brownian CRT and its dislocation measure. 
\end{proof}

This identifies $c_n^\prime=\sqrt{\nu_0}\sqrt{n}/\sigma$. Note that Step 3. of our programme therefore is to show, for Galton-Watson trees $\cG^{(n)}$ with 
$n$ leaves, joint GP convergence in distribution of \vspace{-0.2cm}
$$\left(\frac{\sqrt{\nu}_0}{\sqrt{n}}\sigma\cG^{(n)},\frac{1}{\sqrt{\nu_0}\sqrt{n}}\sigma\modcut(\cG^{(n)})\right)\rightarrow(\cT_{\rm Br},{\rm cut}(\cT_{\rm Br})).\vspace{-0.2cm}$$
For Galton-Watson trees $\cG^{(n)}_V$ with $n$ vertices, Bertoin, Miermont and Dieuleveut showed\vspace{-0.2cm}
$$\left(\frac{1}{\sqrt{n}}\sigma\cG^{(n)}_V,\frac{1}{\sqrt{n}}\frac{1}{\sigma}{\rm cut}_{\rm BM}(\cG^{(n)}_V)\right)\rightarrow(\cT_{\rm Br},{\rm cut}(\cT_{\rm Br})),\vspace{-0.1cm}$$
$$\left(\frac{1}{\sqrt{n}}\sigma\cG^{(n)}_V,\frac{1}{\sqrt{n}}\left(\sigma+\frac{1}{\sigma}\right){\rm cut}_{\rm D}(\cG^{(n)}_V)\right)\rightarrow(\cT_{\rm Br},{\rm cut}(\cT_{\rm Br})).$$
We understand the appearance of $\sqrt{\nu_0}$, which is simply due to the different conditioning: $\cG^{(n)}$ is conditioned to have $n$ leaves, while $\cG^{(n)}_V$ is conditioned to have $n$ vertices. Dieuleveut gave a heuristic interpretation of her factor $\sigma+1/\sigma$, comparing to just $1/\sigma$ for Bertoin and Miermont by referring to the fact that the number of edges removed in a vertex fragmentation is $k\ge 2$, so she gets $\sum_kk\nu_k\times k=\sigma^2+1$ as the speed-up compared to Bertoin and Miermont. Here, the first $k$ in the sum reflects the fact that a branch point with $k$ children is selected with probability proportional to $k$. This is what
we have changed. Therefore, the average number of edges we remove is smaller when we drop the rate to being proportional to $k-1$, and we get $\sum_k(k-1)\nu_k\times k=\sigma^2$ as the speed-up compared to Bertoin and Miermont.

We deduce  $({\rm cut}_{\rm D}(\widehat{\cG}^{(n)})/c_n^\prime,\modcut(\cG^{(n)})/c_n^\prime,\ourcut(\cG^{(n)})/c_n^\prime)\rightarrow(\cT_{\rm Br},\cT_{\rm Br},\cT_{\rm Br})$ in GH$^3$ by Lemma \ref{lmcut} and Proposition \ref{propcut} and since GHP convergence implies GH convergence, completing Step 1.\ for finite variance offspring distribution. Step 1.\ for offspring distributions in the domain of an infinite variance stable distribution is beyond the scope of this paper. The interested reader is referred to 
\cite[Section 5.2]{HM10}, where Haas and Miermont establish the invariance principle for infinite-variance Galton-Watson trees using their convergence criterion. Their 
arguments would need to be adapted to cut-trees with splitting rule given in Proposition \ref{propspl}.  

\subsection{Step 2: Coding function convergence of modified Galton-Watson trees}

Given a rooted planar tree ${\bf t}$, recall that  $\zeta({\bf t})$ and $\lambda({\bf t})$ denote the total number of vertices and leaves of ${\bf t}$, respectively. Define the Lukasiewicz path, contour function and height function, denoted by ${\cal X}({\bf t}), {\cal C}({\bf t}), {\cal H}({\bf t})$, as follows.  To define ${\cal C} ({\bf t})$, consider a particle that visits the tree in planar order, starting from the root and moving continuously at unit speed up and down the edges of unit length, for each branch point exploring the subtrees in the (left to right) planar order. Then for $s\in [0, 2\zeta({\bf t})]$,  let ${\cal C}_s({\bf t})$ be the distance of the particle  to the root at time $s$.  To define ${\cal X}({\bf t})$ and  ${\cal H}({\bf t})$,  let $\{{v}_j(\ft)\colon j=0,1 ,\ldots, \zeta({\bf t})-1\}$ be the vertices of ${\bf t}$ in the order encountered by $\cC(\ft)$, without duplication.  The height function ${\cal H}({\bf t})$ is defined by letting ${\cal H}_j({\bf t})$ be the generation or \em height \em $|v_j(\ft)|$ of vertex $v_j(\ft)$. The  Lukasiewicz path is defined by ${\cal X}_0({\bf t})=0$ and
$$
{\cal X}_{j+1}({\bf t})={\cal X}_j({\bf t})+k_{v_j(\ft)}(\ft)-1,\quad j=0,\ldots, \zeta({\bf t})-1,
$$
where $ k_{v_j(\ft)}(\ft)$ is the number of children of $v_j(\ft)$ in $\ft$.
Further denote by 
$$\Lambda_0(\ft)=0,\qquad\Lambda_k({\bf t})=\#\{j\le k\colon{\cal X}_j({\bf t})-{\cal X}_{j-1}({\bf t})=-1\},\quad 1\le k\le\zeta({\bf t}),$$ 
the leaf counting process of ${\bf t}$.

Let ${\cal G}^{(n)}$ be a critical Galton-Watson tree with $n$ leaves. We recall from \cite[Theorem 8.1]{Kor12} and \cite[Theorem 3.3]{Kor14} the invariance principle 
for Galton-Watson trees in terms of coding functions, expressed as a joint convergence on the Skorokhod space ${\rm Sk}$ of c\`adl\`ag functions on $[0,1]$.
\begin{prop}
\label{TheK12} In the setting of Section \ref{sec2.1}, we have  
\begin{equation}
 \sup_{0\leq t\leq 1}\left|\frac{\Lambda_{[ \zeta({\cal G}^{(n)}) t]}({\cal G}^{(n)} )}{ n}-t\right|\vspace{-0.3cm}
 \end{equation}
 together with\vspace{-0.1cm}
 \begin{eqnarray}
\left(\frac{1}{ a_{\zeta( {\cal G}^{(n)})} }{\cal X}_{ [ \zeta( {\cal G}^{(n)}) t]}( {\cal G}^{(n)} ),
\frac{a_{\zeta( {\cal G}^{(n)})}}{ \zeta( {\cal G}^{(n)}) }{\cal C}_{ 2 \zeta( {\cal G}^{(n)}) t}({\cal G}^{(n)} ),
\frac{a_{\zeta( {\cal G}^{(n)})}}{ \zeta( {\cal G}^{(n)}) }{\cal H}_{ [\zeta( {\cal G}^{(n)}) t]}({\cal G}^{(n)} )
\right)_{0\leq t\leq 1}\end{eqnarray}
converge in distribution in $[0,1]\times{\rm Sk}^3$, as $n\rightarrow\infty$, to $(0, X, H, H)$, 
where $X$ is a normalised stable excursion and $H$ is Duquesne and Le Gall's \cite{DuLG} stable height function. If $a_n=\sigma\sqrt{n/2}$, then $H=X=\sqrt{2}B^{\rm ex}$ is a multiple of the normalized excursion $B^{\rm ex}$ of linear Brownian motion.
\end{prop}
Recall that $\widehat{{\cal G}}^{(n)}$ is the modified  tree associated with  ${\cal G}^{(n)}$ as introduced just before Lemma \ref{edgelemma}. Let us be precise and 
extend the planar order of $\cG^{(n)}$ to $\widehat{\cG}^{(n)}$ by placing all extra children to the left. Following ideas of Miermont \cite{Mie08} and de Raph\'elis \cite{deR15}, we introduce the following notation.
Let $\widehat{u}(i)=v_i(\widehat{\cG}^{(n)})$, $i=0, 1,\ldots, \zeta(\widehat{{\cal G}}^{(n)})-1$, and ${u}(j)=v_j(\cG^{(n)})$, $j=0,1, \ldots, \zeta({{\cal G}}^{(n)})-1\}$, be the planar enumerations of $V(\widehat{{\cal G}}^{(n)})$ and $V({{\cal G}}^{(n)})$, respectively.
For $0\le i\le\zeta(\widehat{\cG}^{(n)})-1$, define
\begin{eqnarray}
\varphi_n(i)=j\qquad\mbox{if}\quad
\begin{array}{l}u(j)=\widehat{u}(i)\in {\cal G}^{(n)} \\
 \mbox{or }u(j)\text{ is the parent of }\widehat{u}(i)\notin{\cal G}^{(n)}.
\end{array} 
\end{eqnarray}
This means that $\varphi_n(i)$ is the index of the corresponding vertex in $\cG^{(n)}\subseteq\widehat{\cG}^{(n)}$, or if $\widehat{u}(i)$ is an extra child, $\varphi_n(i)$ is the index of its parent, which will be in $\cG^{(n)}$ as no extra children have offspring.  For $0\le j\le\zeta(\cG^{(n)})-1$, define $\psi_n(j)=\#\{v\in \widehat{{\cal G}}^{(n)}\colon v\prec u(j)\}$,  where $v=\widehat{u}(k)\prec u(j)=\widehat{u}(\varphi_n(i))$ if and only if $k<j$, i.e. $v$ has strictly smaller index in $\widehat{\cG}^{(n)}$ than $u(j)$. Then $\varphi_n(\psi_n(j))=j$, but $\psi_n(\varphi_n(i))=i$ only if $\widehat{u}(i)\in\cG^{(n)}$.

Hence, the function
$\left(\frac{1}{\zeta(\widehat{{\cal G}}^{(n)} )}\psi_n([\zeta({\cal G}^{(n)} ) t ]), 0\leq t\leq1\right)$ can be regarded as the right inverse of 
$\left(\frac{1}{\zeta({\cal G}^{(n)} )}\varphi_n([\zeta(\widehat{{\cal G}}^{(n)} ) t ]), 0\leq t\leq1\right)$, in the approximate sense that their composition is the step 
function with steps $1/\zeta(\cG^{(n)})$ at times $j/\zeta(\cG^{(n)})$, $1\le j\le\zeta(\cG^{(n)})$. As $\zeta(\cG^{(n)})\ge\lambda(\cG^{(n)})=n$ will tend to infinity, 
this composition will approach the identity on $[0,1]$ uniformly in $t\in[0,1]$.
\begin{figure}[t]\centering
\includegraphics[scale=0.57]{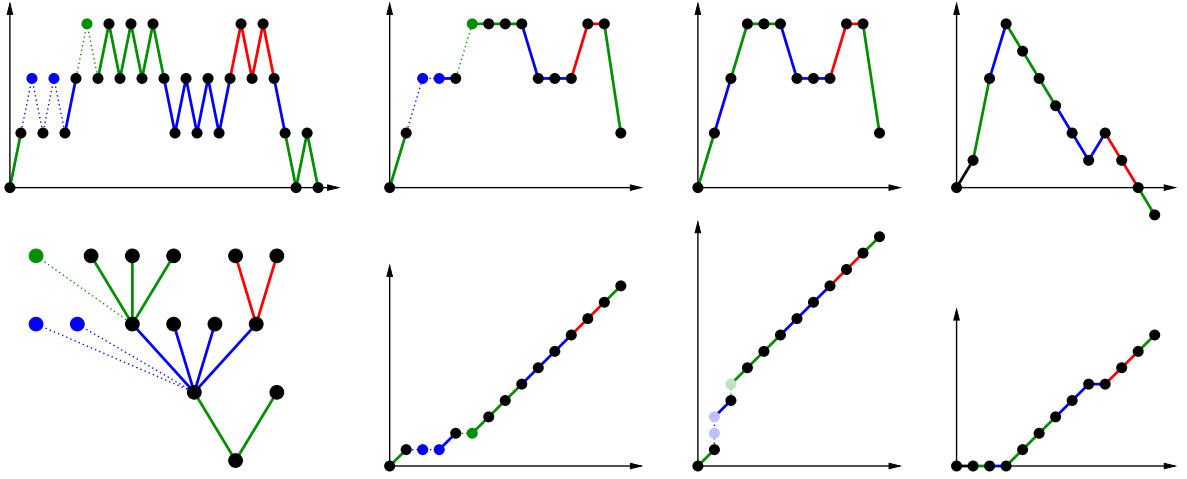}
  \caption{Illustration for some $\cG^{(n)}$ with $n=8$. Top row: ${\cal C}(\widehat{\cG}^{(n)})$, ${\cal H}(\widehat{\cG}^{(n)})$, ${\cal H}(\cG^{(n)})$, ${\cal X}(\cG^{(n)})$, bottom row: $\widehat{\cG}^{(n)}$, $\varphi_n$ and $\psi_n$ and $\Lambda(\cG^{(n)})$. Coloured lines only illustrate how functions relate.}
    \label{fig2}\vspace{-0.2cm}
\end{figure}
\begin{prop}
\label{PropStep2} In the setting of the previous theorem, with $\widetilde{a}_n=a_n/\nu_0^{1/\alpha}$, we also have joint convergence in distribution in $[0,1]^2\times{\rm Sk}^3$ of 
$$\sup_{0\leq t\leq 1}\left|\frac{\Lambda_{[ \zeta( {\cal G}^{(n)}) t]}({\cal G}^{(n)} )}{ n}-t\right|,\qquad \sup_{0\leq t\leq 1}\left|\frac{\varphi_n([\zeta(\widehat{{\cal G}}^{(n)} ) t ])}{\zeta({\cal G}^{(n)} )}-t\right|,$$
together with
 \begin{equation}
\left(\frac{1}{\widetilde{a}_n}{\cal X}_{ [ \zeta( {\cal G}^{(n)}) t]}( {\cal G}^{(n)} ),
\frac{\widetilde{a}_n}{n}{\cal C}_{ 2 \zeta( {\cal G}^{(n)}) t}({\cal G}^{(n)} ),
\frac{\widetilde{a}_n}{n}{\cal H}_{  [\zeta( \widehat{{\cal G}}^{(n)} ) t]}( \widehat{{\cal G}}^{(n)}  )
\right)_{0\leq t\leq 1}\end{equation}
to $(0,0,X,H,H)$.
\end{prop}

The proof will be based on the following lemma.

\begin{lm}
\label{lemTime} We have
$$\sup_{0\leq t\leq 1}\left|\frac{1}{\zeta({\cal G}^{(n)} )}\varphi_n([\zeta(\widehat{{\cal G}}^{(n)} ) t ])-t\right|\rightarrow0,\quad \text{in probability, as }n\rightarrow\infty.$$
\end{lm}
\begin{proof}
According to the definition of $\psi_n$, and by the convention that extra children are placed to the left of other children (and hence enumerated first), we have for any $\ell<\zeta({\cal G}^{(n)} )$,
$$
\psi_n(\ell+1)=\Lambda_{\ell+1} ({\cal G}^{(n)} )+\sum_{j\leq \ell}1_{\{u(j)\in {\rm Br}({\cal G}^{(n)} )\}}+\sum_{j\leq \ell}1_{\{u(j)\in {\rm Br}({\cal G}^{(n)} )\}}(k_{u(j)}(\cG^{(n)})-2).
$$
Meanwhile, by definition of the Lukasiewicz path, we have 
$$
\sum_{j\leq \ell}1_{\{u(j)\in {\rm Br}({\cal G}^{(n)} )\}}(k_{u(j)}(\cG^{(n)})-1)- \Lambda_{\ell} ({\cal G}^{(n)} )={\cal X}_{\ell}({\cal G}^{(n)} ).
$$
Thus,
$$
\psi_n(\ell+1)= \Lambda_{\ell+1}({\cal G}^{(n)} )+ \Lambda_{\ell}({\cal G}^{(n)} )+{\cal X}_{\ell}({\cal G}^{(n)} ).
$$
Note that $a_n=o(n)$. Since $\zeta(\widehat{{\cal G}}^{(n)} )=2n-1$, one can immediately see from Proposition \ref{TheK12} that
$$\sup_{0\leq t\leq 1}\left|\frac{1}{\zeta(\widehat{{\cal G}}^{(n)} )}\psi_n([\zeta({\cal G}^{(n)} ) t ])-t\right|\longrightarrow0,\quad \text{in probability}.$$
By definition of $\varphi_n$ and $\psi_n$, one sees $\varphi_n(\psi_n(k))=k$.  So for fixed $ t\in [0, 1]$, as $ n\rightarrow \infty$,
$$\frac{1}{\zeta({\cal G}^{(n)} )}\varphi_n([\zeta(\widehat{{\cal G}}^{(n)} ) t ])\longrightarrow t,\quad\text{ in probability.}$$
Since $t\mapsto  \varphi_n([\zeta(\widehat{{\cal G}}^{(n)} ) t ])$ is non-decreasing for each $n \geq 1$, Dini's Theorem yields 
$$\left(\frac{1}{\zeta({\cal G}^{(n)} )}\varphi_n([\zeta(\widehat{{\cal G}}^{(n)} ) t ]),\quad 0\leq t\leq 1\right) {\longrightarrow} \left(t,\, 0\leq t\leq 1\right)\quad\text{ in distribution}.$$
 And hence the desired result holds since the identity function is deterministic and continuous. \end{proof}
\begin{rem}\rm
The tree $\widehat{\cal G}^{(n)}$ can be regarded as a 2-type Galton-Watson tree. The analogue of Lemma \ref{lemTime} was obtained by Miermont \cite{Mie08}  for irreducible and non-degenerate multi-type Galton-Watson trees  under a ``small exponential moment'' condition; see Lemma 6 and the proof of Theorem 2 there.
\end{rem}

\noindent {\it Proof of Proposition \ref{PropStep2}.}  We see that
$$
\left|{\cal H}_{  \varphi_n([\zeta( \widehat{{\cal G}}^{(n)}) t])}({\cal G}^{(n)})-{\cal H}_{  [\zeta( \widehat{{\cal G}}^{(n)} ) t]}( \widehat{{\cal G}}^{(n)}  )\right|\leq1.
$$
Thus with Lemma \ref{lemTime} and Proposition \ref{TheK12}, we obtain as $n\rar\infty$,
$$
\left(\frac{a_{\zeta( {\cal G}^{(n)} )}}{ \zeta( {\cal G}^{(n)} ) }{\cal H}_{  \varphi_n(\zeta( \widehat{{\cal G}}^{(n)}) t)}({\cal G}^{(n)})\right)_{0\leq t\leq 1}\overset{d}{\longrightarrow}H.
$$
Meanwhile, by \cite[Lemma 2.7]{Kor12}, we have
$$
\frac{\zeta({\cal G}^{(n)}) }{n}\longrightarrow \frac{1}{\nu_0},
$$
in distribution and hence in probability. Then a standard argument based on the Skorokhod representation theorem establishes the desired result. \qed

\bigskip

Since uniform convergence of either height functions or contour functions implies GHP convergence, this completes Step 2.\ with $c_n=n/(\sqrt{2}\widetilde{a}_n)$, not 
just in the finite-variance case with $c_n=\sqrt{n}/(\sigma\sqrt{\nu_0})$ by Lemma \ref{lem2.1}, but also for offspring distributions in the stable domain of attraction. 
In fact, the convergence of Lukasiewicz paths of $\widehat{\cG}^{(n)}$ can be proved similarly.

\subsection{Step 3: Joint GP convergence of the modified tree and its cut-tree}

In the sequel, we mainly have the case of a finite-variance modified Galton-Watson tree in mind, but we include the stable case, where the argument is the same. From 
here, we follow Dieuleveut \cite[Section 4]{Die13} closely (and \cite[Section 2]{Die13} for the stable case, which contains some of the details also needed for the 
finite variance case).
Let $\zeta_n=\zeta(\cG^{(n)})$ and $\widehat{\zeta}_n=\zeta(\widehat{\cG}^{(n)})$. Also write 
$(X^{(n)},H^{(n)},C^{(n)})$ for suitably scaled Lukasiewicz path $\cX(\cG^{(n)})$, height function $\cH(\cG^{(n)})$ and contour function $\cC(\cG^{(n)})$, $n\ge 1$, 
which converge to the corresponding triplet $(X,H,H)$ associated with a stable tree $\cT$ (including the Brownian CRT, in which case $X=H=\sqrt{2}B^{\rm ex}$). 
\begin{lm}[cf. \cite{Die13} Lemmas 2.4, 4.2]\label{lmrev} If $(H^{(n)},C^{(n)},X^{(n)})\rightarrow(H,H,X)$ in distribution in ${\rm Sk}^3$, then $(H^{(n)},X^{(n)},\widetilde{X}^{(n)})\rightarrow (H,X,\widetilde{X})$ in distribution in ${\rm Sk}^3$,
  where $\widetilde{X}^{(n)}$ and $\widetilde{X}$ are Lukasiewicz paths with all orders of children reversed.
\end{lm}
\begin{proof} Dieuleveut's argument only uses the identical distribution of reversed quantities $(\widetilde{X},\widetilde{C})$, the fact that $\widetilde{X}$ is a
  measurable function of the jump sizes and jump times of $X$ to identify the limit in the stable case, and the symmetry 
  $\widetilde{C}^{(n)}_t=C^{(n)}_{1-t}$ and continuity of $H$ to identify the limit in the case of a Brownian limit. Hence, her argument also establishes this analogous
  result.
\end{proof}
\begin{lm}[cf. \cite{Die13} Lemmas 2.7, 2.8, 4.3, 4.4]\label{lmlimits} Let $(X^{(n)},H^{(n)},\widetilde{X}^{(n)},U^{(n)})\rightarrow(X,H,\widetilde{X},U)$ almost surely, for some
  $U^{(n)}=(U^{(n)}_i,i\ge 1)$ and $U=(U_i,i\ge 1)$ with $U^{(n)}_i\in\{\frac{j}{\zeta_n},1\le j\le\zeta_n\}$, and i.i.d. $U_i\sim{\rm Unif}(0,1)$ independent of 
  $(X,H,\widetilde{X})$. Then we also have the following limits.
  \begin{itemize}
    \item The shape of the subtree $\cR^{(n)}(k)$ of $\cG^{(n)}$ spanned by $0,U_1^{(n)},\ldots,U_k^{(n)}$ is constant a.s.\
      for $n$ large enough, equal to the shape $R(k)$ of the subtree $\cR(k)$ of $\cT$ spanned by $0,U_1,\ldots,U_k$.
    \item For every edge $e=(v\rightarrow v^\prime)\in E(R(k))$, denote by $e^{(n)}_+(k),e^{(n)}_-(k)\in V(\cR^{(n)}(k))$ the vertices corresponding to $v=e^+(k)$ and $v^\prime=e^-(k)$,
      and by
      $V_e^{(n)}(k)$ the set of vertices between $e^{(n)}_+(k)$ and $e^{(n)}_-(k)$. Then the rescaled lengths of the edge converge a.s.:
      $$\frac{\widetilde{a}_n}{n}\left(1+\#V_e^{(n)}(k)\right)=H_{b_n(e^{(n)}_+(k))}^{(n)}-H_{b_n(e^{(n)}_-(k))}^{(n)}\rightarrow H_{b(e_+(k))}-H_{b(e_-(k))},$$
      where $b_n(w)$ is the first time of $H^{(n)}$ corresponding to $w\in V(\cG^{(n)})$, similarly $b(w)$, $w\in\cT$.
    \item For every branch point $v\in{\rm Br}(R(k))$, rescaled numbers of children converge a.s., i.e.
      $$\frac{1}{\widetilde{a}_n}k_v(\cG^{(n)})\sim\frac{1}{\widetilde{a}_n}\left(k_v(\cG^{(n)})-1\right)=\Delta X_{b_n(v)}^{(n)}\rightarrow\Delta X_{b(v)},$$
      which vanishes in the finite-variance case.
    \item For every edge $e\in E(R(k))$, sums of rescaled numbers of children converge a.s., as follows:
      \begin{align*}\frac{1}{\widetilde{a}_n}\sum_{v\in V_e^{(n)}(k)}\left(k_v(\cG^{(n)})-1\right)
		&\rightarrow (X_{b(e^+)}+\widetilde{X}_{\widetilde{b}(e^+)})-(X_{b(e^-)}+\widetilde{X}_{\widetilde{b}(e^-)})-\Delta X_{b(e^-)},
      \end{align*}
      which in the finite-variance case simplifies to $H_{b(e^+)}-H_{b(e^-)}$. If we replace $(k_v(\cT_n)-1)$ by $k_v(\cT_n)$, we get the same limit
      in the stable case, while in the finite-variance case, $n/a_n^2=1/\sigma^2$ and we obtain a limit $(1+1/\sigma^2)(H_{b(e^+)}-H_{b(e^-)})$ instead.
  \end{itemize}
\end{lm}
\begin{proof} Dieuleveut's arguments are entirely deterministic, just requiring the limiting random variables to avoid certain degeneracies a.s.
\end{proof}
To apply this, take independent $(U_i,i\ge 1)$ and use Skorokhod's representation theorem to have the convergences of Proposition \ref{PropStep2} and Lemma \ref{lmrev} 
jointly and almost surely. We now use $U_i$ to sample a uniform edge in $\widehat{\cG}^{(n)}$ and take as $U^{(n)}_i$ the corresponding time of
$H^{(n)}$, i.e. $U^{(n)}_i=\varphi_n([(2n-2)U_i]+1)/\zeta(\cG^{(n)})$, which is not independent of $\cG^{(n)}$, but since $\varphi_n$, 
converges uniformly to the identity on
$[0,1]$, the almost sure convergence needed to apply Lemma \ref{lmlimits} holds, with limit $U_i$ independent of the limiting coding functions.

\begin{prop}[cf. \cite{Die13} Propositions 2.5 and 4.1]\label{prop4.1} Consider edge samples $\xi_n(i)$ in $\widehat{\cG}^{(n)}$ and the continuous-time Dieuleveut 
  vertex fragmentation of $\widehat{\cG}^{(n)}$ that removes the edges above vertex $v\in{\rm Br}(\widehat{\cG}^{(n)})$ at rate $k_v(\widehat{\cG}^{(n)})/2\widetilde{a}_n$.
  Define mass processes $(\mu_{n,\xi_n(i)}(t))_{t\ge 0}$ capturing the evolution of the proportion of leaves in the component containing $\xi_n(i)$, $i\ge 1$, and   
  separation times $\tau_n(i,j)$ of $\xi_n(i)$ and $\xi_n(j)$, $i,j\ge 1$. Then in {\rm GP}$\times[0,\infty)^\bN\times{\rm Sk}^\bN$, in distribution, as $n\rightarrow\infty$,
  $$\left(\frac{\widetilde{a}_n}{n}\widehat{\cG}^{(n)},(\tau_n(i,j))_{i,j\ge 1},(\mu_{n,\xi_n(i)}(t))_{t\ge 0,i\ge 1}\right)\rightarrow\left(\cT,(c^{-1}\tau(i,j))_{i,j\ge 1},(\mu_{\xi(i)}(ct))_{t\ge 0,i\ge 1}\right),$$
  where $c=1$ in the stable case and/or when rates are proportional to $k-1$, while it is $c=1+1/\sigma^2$ only in the finite variance case when rates are
  proportional to $k$.
\end{prop}
\begin{proof} Dieuleveut's arguments work since we can still sample $\xi_n(i)$ from $U_i^{(n)}$ in $[0,1]$, and the remaining arguments only depend on tree convergences and
  rate convergences (up to a factor of $c$), both of which we have, from Proposition \ref{PropStep2}, Lemma \ref{lmrev} and Lemma \ref{lmlimits}.
\end{proof}

The convergences achieved so far imply the convergence of certain modified distances for the discrete cut-trees. These modified distances resemble the Brownian cut-tree
distances and take the following form. We enumerate the $2n-2$ edges of $\widehat{\cG}^{(n)}$ by $1,\ldots,2n-2$ and define for $i,j\in\{1,\ldots,2n-2\}$
$$\delta_n^\prime(0,i)=\int_0^\infty\mu_{n,i}(t)dt\quad\mbox{and}\quad\delta^\prime(i,j)=\int_{t_n(i,j)}^\infty(\mu_{n,i}(t)+\mu_{n,j}(t))dt,$$
where $t_n(i,j)$ is the most recent time when edges $i$ and $j$ were in the same component in the continuous-time vertex fragmentation of $\widehat{\cG}^{(n)}$.

\begin{lm}[cf. \cite{Die13} Lemma 2.1]\label{lm2.1} For all $i,j\in\{1,\ldots,2n-2\}$, we have
  $$\bE\left[\left|\frac{\widetilde{a}_n}{n-1}\delta_n(i,j)-\delta_n^\prime(i,j)\right|^2\right]\le\frac{\widetilde{a}_n}{n-1}\bE\left[\delta_n^\prime(0,i)+\delta_n^\prime(0,j)\right].$$
\end{lm}
\begin{proof} Dieuleveut works conditionally given the tree, so the argument applies to the tree $\widehat{\cG}^{(n)}$ with $2n-2$ edges and the rates 
  $k_v(\widehat{\cG}^{(n)})/2\widetilde{a}_n$ that specify the continuous-time cutting. 
\end{proof}


\begin{lm}[cf. \cite{Die13} Lemma 4.5]\label{lm4.5} Assume $\nu_1=0$ and finite variance $\sigma^2$. Let $\xi_n$ be uniform on $\{1,\ldots,2n-2\}$. Then
  $$\lim_{\ell\rightarrow\infty}\sup_{n\ge 1}\bE\left[\int_{2^\ell}^\infty\mu_{n,\xi_n}(t)dt\right]=0\qquad\mbox{and}\qquad\bE[\delta_n^\prime(0,\xi_n)]\le C_0$$
for some $C_0\in(0,\infty)$.
\end{lm}
\begin{proof}
We use the same ideas from  \cite[Corollary 2]{BM} and   \cite[Lemma 4.5]{Die13} to prove the result. We focus on where the arguments differ. As Dieuleveut pointed out, there is a coupling between vertex-fragmentation and edge-fragmentation by a deterministic procedure. So we directly follow the argument in \cite{BM} by considering uniform edge-cutting on $\widehat{\cal G}^{(n)}$. Recall that $V({\bf t})$ is the set of vertices of ${\bf t}$. For a vertex $u \in V (\widehat{\cal G}^{(n)})$, let $e_u$ be the edge pointing down from $u$ towards the root, and for an edge $e$ of $\widehat{\cal G}^{(n)}$, let $v(e)$ be the vertex such that $e_{v(e)} = e$. Then given $\widehat{\cal G}^{(n)}$, $v(\xi_n)$  is uniform in $V^*(\widehat{\cG}^{(n)})=V(\widehat{\cal G}^{(n)} )\setminus\{\rho\}$. Following Bertoin and Miermont's argument, we obtain
\begin{eqnarray}
{\bE}[n\mu_{n, \xi_n}(t)]&\leq& e^{-t/\sqrt{n}}+{\bE}\left[ \sum_{u\in V^*(\widehat{\cal G}^{(n)})\setminus \{v(\xi_n)\}}e^{-d(u, v(\xi_n))t/\sqrt{n}}\right]\cr
&= &  e^{-t/\sqrt{n}}+\frac{1}{2n-2}{\bE}\left[ \sum_{u, v\in V^*(\widehat{\cal G}^{(n)}),\,u\neq v}e^{-d(u, v)t/\sqrt{n}}\right]\cr
&\leq & e^{-t/\sqrt{n}}+\frac{4}{2n-2}{\bE}\left[ \sum_{u, v\in V^*({\cal G}^{(n)}),\,u\neq v}e^{-d(u, v)t/\sqrt{n}}\right],
\end{eqnarray}
where the last inequality follows from the following observation: for each vertex $v\in V({\cal G}^{(n)})$ with $k_{v}\geq2$ children, say $v_1,\ldots, v_{k_{v}}$, there are $k_{v}-2$ further children in $V(\widehat{\cal G}^{(n)})\setminus V({\cal G}^{(n)})$, say $v_1',\ldots, v_{k_{v}-2}'$. Then for $u\in V(\widehat{\cal G}^{(n)})$, we have $d(u, v_i)+2= d(u, v_i')$ if $v_i$ is an ancestor of $u$; and $d(u, v_i)= d(u, v_i')$ otherwise. If we replace $d(u, v_i')$ by $d(u,v_i)$, then each $ v_i$ would be counted at most twice. We can similarly reduce the sum over $u$ and gain another factor 2.

Denote by $GW_*$ the sigma-finite measure on the space of pointed trees such that
$$
GW_*(\overline{{\bf t}}, v)={\bP}(\overline{\cal G}= \overline{{\bf t}}),
$$
where $\overline{\cal G}$ is the planted version of $\cal G$, with an edge and vertex added below the root, $\overline{{\bf t}}$ denotes a generic planted planar tree and $v\in V({\bf t})$; see Sections 1.2 and 4 in \cite{BM}. Then notice that the set of pointed trees $(\overline{\bf t}, v)$ with exactly $n$ leaves has $GW_{*}$-measure equal to ${\bE}[\zeta({\cal G})1_{\{\lambda({\cal G})=n\}}]\in(0,\infty)$. So the conditional law $GW_{*}(\,\cdot\;|\,\lambda({\bf t})=n)$ on the space of pointed tree with $n$ leaves is well defined and is the same to the distribution of $(\overline{{\cal G}}^{(n)}, \eta)$, where $\eta$ is a uniformly chosen vertex in $V({\cal G}^{(n)} )$. We also note that if $\nu_1=0$, then $ \#V({\cal G}^{(n)} )=\zeta( {\cal G}^{(n)})\leq 2n$. Thus one can deduce that
\beqnn
{\bE}[n\mu_{n, \xi_n}(t)]
&\leq & e^{-t/\sqrt{n}}+\frac{2}{n-1}{\bE}\left[ \sum_{u, v\in V^*({\cal G}^{(n)}),\,u\neq v}e^{-d(u, v)t/\sqrt{n}}\right]\cr
&\leq & e^{-t/\sqrt{n}}+\frac{4n}{n-1}GW_*\left.\left[ \sum_{u\in V({\bf t})\setminus\{v\}}e^{-d(u,v)t/\sqrt{n}}\;\right|\,\lambda({\bf t})=n\right]\cr
&\leq & e^{-t/\sqrt{n}}+\frac{8n}{n-1}\sum_{k\ge 1}e^{-kt/\sqrt{n}}\bE[ \zeta_k({\cal G}^{(n)})],
\eeqnn
where the last inequality follows from the same argument as \cite{BM} by replacing $\#V({\bf t})$ with $\lambda({\bf t})$. Using Lemma \ref{bound}, we obtain
$$
{\bE}(\mu_{n, \xi_n}(t))\leq \frac{e^{-t/\sqrt{n}} }{n}+\frac{4C}{n}\sum_{k\geq1}ke^{-kt/\sqrt{n}}\leq \frac{C^\prime \exp(-t/\sqrt{n})}{n(1-\exp(-t/\sqrt{n}))^2}.
$$
Then it is easy to see that
$$\lim_{l\rightarrow\infty}\sup_{n\geq1}{\mathbb E}\int_{2^l}^{\infty}\mu_{n, \xi_n}(t)dt=0\quad\mbox{and}\quad
\sup_{n\geq1}{\bE}(\delta_n'(\xi_n, 0))=\sup_{n\geq1}\int_0^{\infty} {\bE}(\mu_{n, \xi_n}(t))dt<\infty.
$$
This completes the proof.
\end{proof}

Recall that $\widehat{\cG}^{(n)}$ is the modified Galton-Watson tree conditioned to have $n$ leaves, where the modification is the addition of $k-2$ extra leaves
attached to branch points with $k$ children, for every $k\ge 3$, for every branch point. This tree has $2n-2$ edges and is equipped with the uniform measure on those $2n-2$
edges. Recall further that ${\rm cut}_{\rm D}(\widehat{\cG}^{(n)})$ denotes the Dieuleveut cut-tree of $\widehat{\cG}^{(n)}$. This tree has $2n-2$ leaves and is equipped with the uniform measure on those $2n-2$ leaves, which we enumerate $1,\ldots,2n-2$. Let $c_n=\sqrt{n}/(\sigma\sqrt{\nu_0})$ and
$c_n^\prime=\sqrt{\nu_0}\sqrt{n}/\sigma$.

\begin{thm}[cf. \cite{Die13} Theorem 1.4] If the offspring distribution $\nu$ has finite variance $\sigma^2$,  then we have $\displaystyle\left(\frac{1}{c_n}\widehat{\cG}^{(n)},\frac{1}{c_n^\prime} {\rm cut}_{\rm D}(\widehat{\cG}^{(n)})\right)\longrightarrow\left(\cT_{\rm Br},{\rm cut}(\cT_{\rm Br})\right)$ in distribution in ${\rm GP}^2$, as $n\rightarrow\infty$.
\end{thm}
\begin{proof} With Proposition \ref{prop4.1} and Lemmas \ref{lm2.1} and \ref{lm4.5}, we have provided all ingredients for Dieuleveut's proof to apply to $\widehat{\cG}^{(n)}$.
\end{proof}

It should be possible to approach Lemmas \ref{bound} and \ref{lm4.5} in the stable case and hence complete Step 3. also in the stable case, at least under some 
technical assumptions on the tail of the offspring distribution. Dieuleveut's corresponding 
arguments for her vertex cut-trees in \cite[Section 2.3]{Die13} are by far the most technical part of her paper, spread over 14 pages, and we do not see any new insights 
from adapting them, hence we do not pursue this here.

\begin{appendix}

\section{Appendix: from joint GP to joint GHP convergence}\label{appA}

In a previous version of the present paper, we used an argument that required strengthening joint GP convergence to joint GHP convergence, which may be of independent
interested. This is based on the one-dimensional case of \cite{ALW16}, whose notation and terminology we use here. In particular, $\bX_c$ is the space of weak
equivalence classes of compact metric measure spaces, where an isometry is only required between supports of the measures. In particular, the GHP topology on this space
only requires convergence in the Hausdorff sense of the supports rather than the entire space, see \cite[Section 5]{ALW16}, where also strong equivalence classes are
defined that do apply Hausdorff to the whole space. The set of such strong equivalence classes is denoted by $\mathfrak{X}_c$, the subset of equivalence classes where
the measure has full support is denoted by $\mathfrak{X}_c^{\rm supp}$.

\begin{lm} For any $\bX_c$-valued random variable $\cX$, there is a sequence $\bK_n\subseteq\bX_c$, $n\ge 1$, of Polish subspaces on which ${\rm GP}$ and ${\rm GHP}$ topologies coincide
  and such that $\bP(\cX\in\bK_n)\ge 1-1/n$.
\end{lm}
\begin{proof} \cite[Lemma 3.4]{ALW16} shows that every random compact measured metric space $(X,d,\rho,\mu)$ satisfies the lower mass bound $m_\delta(X)=\inf\{\mu(B(x,\delta))\colon x\in X\}>0$ almost surely, where $B(x,\delta)=\{y\in X\colon d(x,y)\le\delta\}$. In particular, for all $\varepsilon_n=1/n$ and $\delta_k=1/k$ there is $q_{k,n}>0$ such that $\bP(m_{\delta_k}(X)<q_{k,n})\le\varepsilon_n2^{-k}$ so that
$\bP(m_{\delta_k}(X) \ge q_{k,n}\mbox{ for all }k\ge 1)\ge 1-\varepsilon_n$, for all $n\ge 1$. Hence we can define $\bK_n=\{\cX\in\bX_c\colon m_{\delta_k}(\cX) \ge q_{k,n}\mbox{ for all }k\ge 1\}$ and conclude by \cite[Corollary 6.3]{ALW16}.
\end{proof}

Since GP and GHP do not coincide on $\bX_c$, as GHP is Polish and GP is not \cite{ALW16}, this statement is in a sense optimal. Certainly, topologies coincide on $\bigcup_{n\ge 1}\bK_n$ only in the weak sense above, statements about GHP-open sets $A\subseteq\bigcup_{n\ge 1}\bK_n$ would be that $A\cap\bK_n$ is GP-open in $\bK_n$, but $A$ not necessarily open in $\bigcup_{n\ge 1}\bK_n$, so the union $\bigcup_{n\ge 1}A\cap\bK_n$ may or may not be GP-open in $\bX_c$.

\begin{lm}\label{borel} For any $\bX_c$-valued random variable $\cX$, there is a set $\bK\subseteq\bX_c$ with $\bP(\cX\in\bK)=1$ and so that the ${\rm GP}$- and ${\rm GHP}$-Borel $\sigma$-algebras
  on $\bK$ coincide.
\end{lm}
\begin{proof} We have $d_{\rm GP}(\cX,\cX^\prime)\le d_{\rm GHP}(\cX,\cX^\prime)$, so balls satisfy $B_{\rm GHP}(\cX,r)\subseteq B_{\rm GP}(\cX,r)$, hence any GP-open set is also GHP-open, but there may be more GHP-open sets than GP-open sets. But as $m_\delta$ is GP-measurable by \cite[Lemma 3.2]{ALW16}, the sets $\bK_n$ can be chosen GP-measurable and so any GHP-open set $A\subseteq\bK=\bigcup_{n\ge 1}\bK_n$ is such that $A\cap\bK_n$ is GHP-open in $\bK_n$, hence GP-open as topologies coincide on $\bK_n$, hence $A\cap\bK_n$ is GP-measurable in $\bX_c$, hence $\bigcup_{n\ge 1}A\cap\bK_n=A$ is GP-measurable in $\bX_c$. Since all GHP-open sets $A\subseteq\bK$
are GP-measurable, the GHP-Borel $\sigma$-algebra on $\bK$ is included in and hence equal to the the GP-Borel $\sigma$-algebra on $\bK$.
\end{proof}

We want to apply this result to identify limiting distributions that we obtain from GP- and GHP-convergences. The subtlety is that this may fail if the Borel $\sigma$-algebras are different, since the larger GHP-$\sigma$-algebra might then have sets not in the GP-$\sigma$-algebra, on which the distribution could differ.

\begin{lm}\label{lmA3} Let $\cX_n$ and $\cX$ be $\mathfrak{X}_c^{\rm supp}$-valued random variables. Suppose that
  $\cX_n\rightarrow\cX$ in distribution in ${\rm GP}$. Suppose further that the distributions of $\cX_n$, $n\ge 1$, are ${\rm GHP}$-tight. Then $\cX_n\rightarrow\cX$ in distribution in
  ${\rm GHP}$.
\end{lm}
\begin{proof} By \cite[Remark 5.2]{ALW16}, there is a GHP-homeomorphism from $\mathfrak{X}_c^{\rm supp}$ onto $\bX_c$, which is also a GP-homeomorphism. Therefore, we can
  prove the result in $\bX_c$. So, consider any GHP-convergent subsequence $(\cX_{n_k},k\ge 1)$. Then it is GP-convergent, and by uniqueness of limits, the limit must
  be $\cX$. But by the previous lemma, this identifies the GHP-limit. Hence, all convergent subsequences have the same limit, and convergence to that limit holds in GHP. \end{proof}

\begin{cor} Let $(\cX_n,\cX_n^\prime)$ and $(\cX,\cX^\prime)$ be $(\mathfrak{X}_c^{\rm supp})^2$-valued random variables for which we have
$(\cX_n,\cX_n^\prime)\rightarrow(\cX,\cX^\prime)$ in distribution in ${\rm GP}^2$. Suppose also that $\cX_n\rightarrow\cX$ in distribution in ${\rm GHP}$ and
$\cX_n^\prime\rightarrow\cX^\prime$ in distribution in ${\rm GHP}$. Then $(\cX_n,\cX_n^\prime)\rightarrow(\cX,\cX^\prime)$ in distribution in ${\rm GHP}^2$.
\end{cor}
\begin{proof} Since $\cX_n\rightarrow\cX$ and $\cX_n^\prime\rightarrow\cX^\prime$, their joint distributions are GHP$^2$-tight. Any GHP$^2$-convergent subsequence will also
  converge in GP$^2$ and therefore have distributional limit $(\cX,\cX^\prime)$. But Lemma \ref{borel} also implies that the product Borel $\sigma$-algebras of GP$^2$ and GHP$^2$
  coincide. Hence, the GP$^2$ limit again identifies the GHP$^2$-subsequential limit as the distribution of $(\cX,\cX^\prime)$. We conclude as in the proof of Lemma \ref{lmA3}.
\end{proof}

\section{Three constant multiples of the Brownian CRT}\label{appB}

Aldous \cite{Ald91,Ald91b,Ald93} introduced the Brownian CRT $\mathcal{T}_{\rm Ald}$ via the line-breaking construction based on an inhomogeneous Poisson process of rate
$tdt$. Since distances between consecutive points of the Poisson process are lengths in trees, intensity $ctdt$ yields $\mathcal{T}_{\rm Ald}/c$. Aldous's
choice of intensity is such that the convergence of discrete uniform random trees with $n$ vertices labelled $1,\ldots,n$ to $\cT_{\rm Ald}$ is obtained when scaling edges by $\sqrt{n}$. Aldous shows in
\cite[Corollary 22]{Ald93} that $\mathcal{T}_{\rm Ald}$ has the same distribution as the tree $\mathcal{T}_{2B^{\rm ex}}=2\mathcal{T}_{B^{\rm ex}}$, where $\mathcal{T}_{B^{\rm ex}}$ is the tree whose height function is the standard Brownian excursion  $B^{\rm ex}$ of
duration 1. He also shows that $\sigma\mathcal{\cG}_V^{(n)}/\sqrt{n}\rightarrow\mathcal{T}_{\rm Ald}$ in distribution as $n\rightarrow\infty$, when $\cG_V^{(n)}$ is a
Galton-Watson tree with finite variance non-arithmetic offspring distribution conditioned to have $n$ vertices.

By Bertoin \cite{Ber02}, the tree $\cT_{B^{\rm ex}}$ in a Brownian excursion gives rise to a self-similar fragmentation at heights with binary dislocation measure
$\nu_B(dx)=\sqrt{2/\pi}x^{-3/2}(1-x)^{-3/2}1_{[1/2,1)}(x)dx$. In the terminology of \cite{HaM04}, this means that $\cT_{B^{\rm ex}}$ is a self-similar CRT with dislocation measure
$\nu_B$. Haas and Miermont \cite{HM10} reprove Aldous's Galton-Watson convergence result and refer to $\nu_B$ as the Brownian dislocation measure and to $\cT_{B^{\rm ex}}$ as the Brownian continuum random tree, hence their choice is $\cT_{\rm HM}:=\cT_{B^{\rm ex}}=\cT_{\rm Ald}/2$. 

Kortchemski \cite{Kor12} does not use the term ``Brownian CRT'' except when referring to the work of Rizzolo \cite{Riz13} and then without identifying
constants. \cite[Remark 4.6]{Kor12} specifies the height function that encodes his standard limiting tree in the case $\alpha=2$ as $H=\sqrt{2}B^{\rm ex}$. In particular, his
limiting CRT is $\cT_{\rm Kor}:=\sqrt{2}\cT_{B^{\rm ex}}=\sqrt{2}\cT_{\rm HM}=\cT_{\rm Ald}/\sqrt{2}$. Kortchemski's motivation is to align with other stable laws with Laplace
exponent $\lambdaLaplace^\alpha$ and hence with the other stable trees of index $\alpha\in(1,2)$. In this, he follows Duquesne and Le Gall \cite[p.105]{DuLG} and Duquesne
\cite[p.1002]{Duq03}, but they only make qualitative remarks and refer to ``proportional'' when comparing with Brownian excursions, as this is not important for
their results.

Bertoin and Miermont \cite{BM} and Dieuleveut \cite{Die13} use $\cT_{\rm Br}:=\cT_{\rm Ald}=2\cT_{B^{\rm ex}}=2\cT_{\rm HM}=\sqrt{2}\cT_{\rm Kor}$.

\end{appendix}

\subsection*{Acknowledgements}\vspace{-0.2cm}

This work was started during a research visit of the second author to Beijing Normal University. We would like to thank Beijing Normal University and EPSRC (EP/K02979/1) for providing support for this research visit. H.\ He is supported by  NSFC (No.\ 11671041, 11531001, 11371061).

\bibliographystyle{abbrv}
\bibliography{prun}

\end{document}